\documentclass[reqno,a4paper]{amsart}

\usepackage[utf8,latin1]{inputenc}
\usepackage[english]{babel}
\usepackage{eucal,amsfonts,amssymb,amsmath,amsthm,epsfig,mathrsfs}
\usepackage{cancel,soul}
\usepackage{color}
\usepackage{amscd,amsxtra}
\usepackage{enumerate}
\usepackage{latexsym}
\usepackage{setspace}
\usepackage{textcomp}

\allowdisplaybreaks

\newcounter{ipotesi}

\Alph{ipotesi}
 \makeatletter \@addtoreset{equation}{section}

\makeatother \makeatletter

\usepackage[pdfdisplaydoctitle,colorlinks,breaklinks,urlcolor=blue,linkcolor=blue,citecolor=blue]{hyperref}

\newcommand{\scal}[2]{{\left\langle #1,#2\right\rangle}}
\providecommand{\U}[1]{\protect\rule{.1in}{.1in}}

\newtheorem{theorem}{Theorem}[section]

\newtheorem{definition}[theorem]{Definition}
\newtheorem{example}[theorem]{Example}

\newtheorem{lemma}[theorem]{Lemma}

\newtheorem{proposition}[theorem]{Proposition}
\newtheorem{remark}[theorem]{Remark}

\def\N{\mathbb N}

\def\R{\mathbb R}

\newcommand{\K}{\mathcal K}
\newcommand{\J}{{\mathcal{D}}}

\newcommand{\Id}{{\operatorname{Id}}}
\newcommand{\ra}{\rightarrow}
\newcommand{\eps}{{\varepsilon}}

\newcommand{\set}[1]{{\left\{#1\right\}}}

\newcommand{\gen}[1]{{\left\langle #1\right\rangle}}
\newcommand{\abs}[1]{{\left|#1\right|}}
\newcommand{\norm}[1]{{\left\|#1\right\|}}

\newcommand{\eqsys}[1]{{\left\{\begin{array}{ll}#1\end{array}\right.}}
\newcommand{\tc}{\, \middle |\,}

\makeatletter
\@namedef{subjclassname@2020}{%
  \textup{2020} Mathematics Subject Classification}
\makeatother

\begin{document}
\frenchspacing

\title[Differentiability in infinite dimension and Malliavin calculus]{Differentiability in infinite dimension and Malliavin calculus}


\author[D. A. Bignamini, S. Ferrari, S. Fornaro and M. Zanella]{Davide A. Bignamini, Simone Ferrari$^*$, Simona Fornaro, Margherita Zanella}

\address{D. A. Bignamini: Dipartimento di Scienza e Alta Tecnologia (DISAT), Universit\`a degli Studi dell'In\-su\-bria, Via Valleggio 11, 22100 COMO, Italy}
\email{\textcolor[rgb]{0.00,0.00,0.84}{da.bignamini@uninsubria.it}}

\address{S. Ferrari: Dipartimento di Matematica e Fisica ``Ennio De Giorgi'', Universit\`a del Salento, Via per Arnesano snc, 73100 LECCE, Italy}
\email{\textcolor[rgb]{0.00,0.00,0.84}{simone.ferrari@unisalento.it}}

\address{S. Fornaro: Dipartimento di Matematica ``Felice Casorati'', Universit\`a degli studi di Pavia, via A. Ferrata 5, 27100 PAVIA, Italy}
\email{\textcolor[rgb]{0.00,0.00,0.84}{simona.fornaro@unipv.it}}

\address{M. Zanella: Dipartimento di Matematica ``Francesco Brioschi'', Politecnico di Milano, Via E. Bonardi 13, 20133 MILANO, Italy.}
\email{\textcolor[rgb]{0.00,0.00,0.84}{margherita.zanella@polimi.it}}

\thanks{$^*$Corresponding author.}

\begin{abstract}
In this paper we study two notions of differentiability introduced by P. Cannarsa and G. Da Prato (see \cite{CDP96}) and L. Gross (see \cite{GRO1}) in both the framework of infinite dimensional analysis and the framework of Malliavin calculus.
\end{abstract}
       


\date{\today}

\maketitle


 \tableofcontents 

 \section{Introduction}
 
The problem of differentiability along subspaces arises in a natural way in the study of  differential equations for functions of infinitely many variables. 
Over the years, regularity properties along subspaces, such as H\"olderianity and Lipschizianity, have become increasingly central in the theory of infinite dimensional analysis. Various authors have introduced many definitions for these regularity properties, the most widely used are two. One introduced by L. Gross in \cite{GRO1} and systematically presented by V. I. Bogachev in \cite{BOGIE1}, the other one introduced by P. Cannarsa and G. Da Prato in \cite{CDP96} and later developed by E. Priola in his Ph.D. thesis, see \cite{PRI1}. The main purpose of this paper is to compare the Gross and Cannarsa--Da Prato notions of differentiability. 
We will begin by relating the two different notions of gradients along subspaces, introduced in \cite{CDP96} and \cite{GRO1}, when these operators act on a class of sufficiently smooth functions. We will then turn to the specific case where the subspace along which to differentiate is the Cameron--Martin space associated to a given Gaussian measure. In this framework, it is possible to extend such operators to spaces of less regular functions, i.e., Sobolev spaces with respect to the reference Gaussian measure. Such extensions are called Malliavin derivatives. The central result of this paper will be to rigorously show that the Gross and the Cannarsa--Da Prato Malliavin derivatives are two \textit{different} operators (although linked by a relationship that we will clarify) that still have the \textit{same} Sobolev space as their  domain.
This work should therefore be understood as a review of existing results with the specific purpose of relating them through rigorous proofs. Moreover we will also provide the proofs of some results that, to the best of our knowledge, are not present in the literature.


More in details, in Section \ref{sec_2} we recall the notions of differentiability given in \cite{CDP96} and \cite{GRO1} and investigate their relation.
In Subsection \ref{GrossY}, given a separable Hilbert space $H_0$ continuously embedded in a separable Hilbert space $H$, we recall the definition of differentiability along $H_0$ presented by L. Gross in \cite{GRO1} for functions with values in a Banach space $Y$. Over the years, this notion has became essential to prove many regularity results for stationary and evolution equations for functions of infinitely many variables both in spaces of continuous functions and in Sobolev spaces, see for instance \cite{Add2021,ACF20,ACF22,Add-Men-Mir2020,Addo1,Add-Men-Mir2023,ABF21,BF22,BF20,BFS,BFS2,CF18,CL19,LR21}. 
In Subsection \ref{BeppeDif}, given a linear bounded self-adjoint operator $R:H\ra H$ we define the differentiability along the directions of $H_R:=R(H)$ (see Section \ref{BeppeDif}) presented by P. Cannarsa and G. Da Prato in \cite{CDP96}.  This notion of differentiability has also been widely used over the years, see for instance \cite{Add-Ban-Mas2020,Add-Mas-Pri2023, Big21, DA-LU2,DA-LU3, DA-LU-TU1, DA-TU2,GOZ1,MAS2,MAS1, PRI1,Pri99,PZ00}.
When $H_0=H_R$ one can compare the above mentioned notions of differentiability, this is done in Subsection \ref{RvsH} 
where we provide the relationship between the Gross derivatives of order $n$ and the Cannarsa--Da Prato derivatives of order $n$. We highlight that a first comparison between these two derivatives has been already presented in \cite{PRI0} but in the specific case of a injective operator $R$. Subsection \ref{Real} is devoted to the comparison of the above mentioned notions of differentiability with the classical notions of Gateaux and Fr\'echet differentiability. 

The results of Section \ref{sec_2} lay the ground for the comparison of the Malliavin derivatives that naturally appear in the setting considered by L. Gross and P. Cannarsa and G. Da Prato when a Gaussian framework comes into play. This is the content of Section \ref{Mal_top_sec}. On a separable Hilbert space $H$, endowed with its Borel $\sigma$-algebra $\mathcal{B}(H)$, one considers a centered (that is with zero mean) Gaussian measure $\gamma$ with covariance operator $Q$, with $Q:H \rightarrow H$ a linear, self-adjoint, non-negative and trace class operator. The subspace along which to differentiate is the Cameron--Martin space associated to the Gaussian measure, that is $H_0=Q^{1/2}(H)=:H_{Q^{1/2}}$.
It is classical to prove (see e.g. \cite{BOGIE1} and \cite{DaPrato}) that the gradient operators $\nabla_{H_{Q^{1/2}}}$ and $\nabla_{Q^{1/2}}$, in the sense of Gross and Cannarsa--Da Prato, respectively, are closable operators in $L^p(H,\mathcal{B}(H), \gamma)$, $p\ge1$. Their extensions are called Malliavin derivatives and the domain of their extension is a Sobolev space with respect to the measure $\gamma$. We refer to these two Malliavin derivatives as the Malliavin derivative in the sense of Gross and the Malliavin derivative in the sense of Cannarsa--Da Prato, respectively. In Subsections \ref{Bogachev}, \ref{DaPrato} and \ref{fin_rem} we recall the construction of these two Malliavin derivatives and prove that they are indeed two  different operators linked by the relation $\nabla_{H_{Q^{1/2}}}=Q^{1/2}\nabla_{Q^{1/2}}$. Nevertheless these two derivatives, although different, have the same Sobolev space as their  domain.   

In order to make a rigorous comparison between the Malliavin derivative in the sense of Gross and Cannarsa--Da Prato, it is convenient to approach the Malliavin calculus from a more abstract point of view, as done, for example, in \cite{Nualart}. We briefly recall this approach to Malliavin calculus  in Appendix \ref{Malliavin_abstract_sec}.

At first glance it might seem strange to refer to Malliavin derivatives that are different, since one usually speaks of \textit{the} Malliavin derivative. We point out that in fact it would be more appropriate to speak of \textit{a} (choice of) Malliavin derivative rather than \textit{the} Malliavin derivative. In fact, as explained in details in Appendix \ref{Malliavin_abstract_sec}, one can construct infinitely many different Malliavin derivative operators. 
On the other hand, it turns out that all these Malliavin derivatives have the same domain when somehow the Gaussian framework is the same. 
In a sense, the results of Section \ref{Mal_top_sec} can be considered as an example of this general fact in a concrete situation: we deal with two particular Malliavin derivatives that naturally appear in the literature for the study of various problems. However, these Malliavin derivatives are just two possible choices among the infinite possible ones that can be considered in that specific Gaussian framework. 


We emphasize here that the general framework for Malliavin calculus considered in \cite{Nualart} not only proves useful in understanding the relationship between different Malliavin derivatives that appear in the literature in various contexts but also turns out to be particularly flexible for dealing with various problems. We mention, for example, the study of the regularity of solutions to stochastic partial differential equations (see, e.g., \cite{Bally1998,DKN2009,Marinelli2013,Marquez-Carreras2001, Mueller2008, Sard_Nualart, Pardoux1993} for parabolic-type stochastic partial differential equations, \cite{BoZa2,BoZa1} for equations with boundary noise, \cite{MilletSanz, Sar_San2, Sanz_Sardanyons} for the stochastic wave equation, \cite{FerZan, NualartZaidi1997, Morien1999} for fluid-dynamics stochastic partial differential equations, \cite{Cardon-Weber2001} for the stochastic Cahn-Hilliard equation), the study of density formulae and concentration inequalities (see e.g. \cite{NV2009}), the study of ergodic problems (see, among others, \cite{HaiMat}), or even the study of integration by parts formulas on level sets in infinite dimensional spaces (see, e.g., \cite{Addo1,BDPT,BoTuZa,DPLT2}). Moreover, there are applications to finance, see e.g. \cite{AL21}, and to numerical analysis (see, e.g., \cite{Talay1,Talay2, Crisan,Roz}).


Section \ref{Lasry} should be interpreted as an application of the results of Section \ref{sec_2}. We establish an interpolation result (see Theorem \ref{intDS}). In \cite[Section 3]{BFS} and \cite[Proposition 2.1]{CDP96}, two interpolation results analogous to Theorem \ref{intDS} are proven. The one in \cite[Proposition 2.1]{CDP96} is in the sense of Cannarsa--Da Prato differentiability, while the one in \cite[Section 3]{BFS} is in the sense of Gross differentiability. Theorem \ref{intDS} covers the degenerate case, which is not included in \cite[Section 3]{BFS} and \cite[Proposition 2.1]{CDP96} (see Remark \ref{dege}). This improvement is possible due to some regularity results about Lasry--Lions type approximants that are finer than those found in the literature (see, for example, \cite{BFS,CDP96}). These results can be found in Section \ref{Mourinho} and are of interest regardless of Theorem \ref{intDS}. Finally, we recall that interpolation theorems are useful for Schauder regularity results for Ornstein--Uhlenbeck type operators in infinite dimensions, see, for instance, \cite{BFS,CDP12,CL19,DaPrato13,PZ00}.

\subsection*{Notations}

In this section we recall the standard notations that we will use throughout the paper. We refer to \cite{FHH11} and \cite{RE-SI1} for notations and basic results about linear operators and Banach spaces. Throughout the paper, all Banach and Hilbert spaces are supposed to be real.

Let $\K_1$ and $\K_2$ be two Banach spaces equipped with the norms $\norm{\cdot}_{\K_1}$ and $\norm{\cdot}_{\K_2}$, respectively. Let $H$ be a Hilbert space equipped with the inner product $\scal{\cdot}{\cdot}_{H}$ and associated norm $\norm{\cdot}_H$.

For any $k\in\N$, let $\mathcal{L}^{(k)}(\K_1;\K_2)$ be the space of continuous multilinear mappings from $\K_1^k$ to $\K_2$ endowed with the norm
\[\norm{T}_{\mathcal{L}^{(k)}(\K_1;\K_2)}:=\sup_{\substack{h_1,\ldots,h_k\in \K_1\setminus\set{0}}}\frac{\|T(h_1,\ldots,h_k)\|_{\K_2}}{\|h_1\|_{\K_1}\cdots\|h_k\|_{\K_1}}.\]
If $\K_2=\K_1$ we use the notation $\mathcal{L}^{(k)}(\K_1)$. If $k=1$ we write $\mathcal{L}(\K_1;\K_2)$ and $\mathcal{L}(\K_1)$, respectively. If $\K_2=\R$ and $k=1$ we use the standard notation $\K_1^*:=\mathcal{L}(\K_1;\R)$ to denote the topological dual of $\K_1$. By convention we set $\mathcal{L}^{(0)}(\K_1):=\K_1$. We denote by $\Id_{\K_1}$  the identity operator from $\K_1$ to itself. 

We say that $Q\in\mathcal{L}(H)$ is a non-negative (positive) operator if 
\(
\gen{Qx,x}_H\geq 0\ (>0),
\)
for every $x\in H\setminus\set{0}$.
$Q \in \mathcal{L}(H)$ is a non-positive (negative) operator if the operator $-Q$ is non-negative (positive). Let $Q\in\mathcal{L}(H)$ be a non-negative and self-adjoint operator. We say that $Q$ is a trace class operator, if
\begin{align}\label{SantaClaus}
{\rm Trace}[Q]:=\sum_{n=1}^{+\infty}\scal{Qe_n}{e_n}_H<+\infty,
\end{align}
for some (and hence for all) orthonormal basis $\{e_n\}_{n\in\N}$ of $H$. We recall that the definition of trace is independent of the choice of the orthonormal basis in \eqref{SantaClaus}. 

We denote by $\mathcal{B}(\K_1)$  the family of the Borel subsets of $\K_1$. $B_b(\K_1;\K_2)$ is the set of the bounded and Borel measurable functions from $\K_1$ to $\K_2$. If $\K_2=\R$ we simply write $B_b(\K_1)$. $C_b(\K_1;\K_2)$ (${\rm BUC}(\K_1;\K_2)$, respectively) is the space of bounded and continuous (uniformly continuous, respectively) functions from $\K_1$ to $\K_2$. If $\K_2=\R$ we write $C_b(\K_1)$ (${\rm BUC}(\K_1)$, respectively). Both $C_b(\K_1;\K_2)$ and ${\rm BUC}(\K_1;\K_2)$ are Banach spaces if endowed with the norm
\[
\norm{f}_{\infty}=\sup_{x\in \K_1}\|f(x)\|_{\K_2}.
\]

\section{Differentiability along subspaces}
\label{sec_2}



In this section we present the notions of differentiability along subspaces considered in \cite{CDP96} and \cite{GRO1}. In Subsection \ref{GrossY}, we present the notion of differentiability along a Hilbert subspace first considered by L. Gross in \cite{GRO1,KUO1} for vector valued functions. 
This notion often appears in the literature, in particular in the study of transition semigroups associated with stochastic partial differential equations. For example, in \cite{ABF21}, a Harnack-type inequality is investigated. In \cite{BF20, BFS, CL19, LR21}, results regarding Schauder regularity are explored, and in \cite{ACF20, ACF22, BF22, CF16, CF18}, the Sobolev theory is examined. It is also worth mentioning \cite{Add-Men-Mir2020, Addo1, Add-Men-Mir2023}, where integration by parts formulas on open convex domains are studied.

For the sake of clarity, in Subsubsection \ref{Gross_real} we rewrite some definitions of Subsection \ref{GrossY} in the special case of real-valued functions.
In Subsection  \ref{BeppeDif} we recall the notion of differentiability along a particular subspace $H_0$ of a Hilbert space $H$ given by P. Cannarsa and G. Da Prato in \cite{CDP96} and later revised by E. Priola in \cite[Sections 1.2 and 1.3]{PRI1}. 
This approach is also widely employed in the literature. For example, in \cite{FE-GO2, FE-GO1, GOZ1, MAS2, MAS1, Pri99, PZ00}, it is applied to study the regularity properties of transition semigroups in Banach spaces. Additionally, in \cite{Add-Mas-Pri2023, MA-PR2, MA-PR1}, applications to the regularization by noise theory can be found.

Subsection \ref{RvsH} focus on the comparison between the two above mentioned notions of differentiability.
Finally, in Subsection \ref{Real} we make clear their relation with the classical notions of Fr\'echet and Gateaux differentiability.

\subsection{Differentiability in the sense of Gross}\label{GrossY}
Here we introduce the notion of Gross differentiability. 
We thought it appropriate to prove some results concerning differentiability in the sense of Gross in a rather general setting, since these results are used in many papers \cite{ACF20,ACF22,ABF21,BF22,BF20,BFS,CF16,CF18,CL19,LR21}. 
Throughout this subsection $X$ and $Y$ will denote two separable Banach spaces endowed with the norm $\norm{\cdot}_X$ and $\norm{\cdot}_Y$, respectively, and $H_0$ will denote a separable Hilbert space equipped with the inner product $\langle\cdot,\cdot\rangle_{H_0}$ and associated norm $\norm{\cdot}_{H_0}$. We assume $H_0$ to be continuously embedded in $X$, namely there exists $C>0$ such that 
\begin{equation}\label{immercontinua}
\norm{h}_X\leq C\norm{h}_{H_0}, \qquad h \in H_0.
\end{equation}

Let us start by recalling the notions of $H_0$-continuity and $H_0$-Lipschitzianity.
\begin{definition}
We say that a function $\varphi:X\ra Y$ is $H_0$-continuous at $x\in X$ if
\begin{align*}
\lim_{\norm{h}_{H_0}\ra 0}\norm{\varphi(x+h)-\varphi(x)}_Y=0.
\end{align*}
$\varphi$ is $H_0$-continuous if it is $H_0$-continuous at any point $x\in X$.
We say that $\varphi:X\ra Y$ is \emph{$H_0$-Lipschitz} if there exists a positive constant $L_{H_0}$ such that for any $x\in X$ and $h\in H_0$ it holds
\begin{align}\label{H-Lip}
\|\varphi(x+h)-\varphi(x)\|_Y\leq L_{H_0}\norm{h}_{H_0}.
\end{align}
The infimum of all the possible constants $L_{H_0}$ appearing in \eqref{H-Lip} is called $H_0$-Lipschitz constant of $\varphi$.
\end{definition}

We now introduce the notions of Fr\'echet and Gateaux differentiability along $H_0$.
\begin{definition}\label{H-frechet}
We say that a function $\varphi:X\ra Y$ is $H_0$-Fr\'echet differentiable at $x\in X$ if there exists $L_x\in \mathcal{L}(H_0;Y)$ such that 
\[
\lim_{\norm{h}_{H_0}\ra 0}\frac{\|\varphi(x+h)-\varphi(x)-L_xh\|_Y}{\norm{h}_{H_0}}=0.
\]
The operator $L_x$ is unique and it is called $H_0$-Fr\'echet derivative of $\varphi$ at $x\in X$. We set $\J_{H_0} \varphi(x):=L_x$. We say that $\varphi$ is $H_0$-Fr\'echet differentiable if it is $H_0$-Fr\'echet differentiable at any point $x\in X$.

We say that $\varphi$ is twice $H_0$-Fr\'echet differentiable at $x\in X$ if $\varphi$ is $H_0$-Fr\'echet differentiable and the mapping $\J_{H_0} \varphi:X\ra \mathcal{L}(H_0;Y)$ is $H_0$-Fr\'echet differentiable. We call second order $H_0$-Fr\'echet derivative of $\varphi$ at $x\in X$ the unique $\J_{H_0}^2 \varphi(x)\in \mathcal{L}^{(2)}(H_0;Y)$ defined by
\[
\J^2_{H_0} \varphi(x)(h,k):=\J_{H_0} (\J_{H_0} \varphi(x)h)k,\quad h,k\in H_0.
\]
\noindent
In a similar way, for any $k\in\N$ we introduce the notion of $k$-times $H_0$-Fr\'echet differentiability of $\varphi$ and we denote by $\J_{H_0}^k \varphi(x)$ its $k$-order $H_0$-Fr\'echet derivative at $x\in X$. In particular $\J_{H_0}^k \varphi(x)$ belongs to $\mathcal{L}^{(k)}(H_0;Y)$.
We say that $\varphi$ is $k$-times $H_0$-Fr\'echet differentiable if it is $k$-times $H_0$-Fr\'echet differentiable at any point $x\in X$.
\end{definition}

\begin{definition}\label{H-Gateaux}
We say that a function $\varphi:X\ra Y$ is $H_0$-Gateaux differentiable at $x\in X$ if there exists $L_x\in\mathcal{L}(H_0;Y)$ such that for any $h\in H_0$ 
\[
\lim_{s\ra 0}\norm{\frac{\varphi(x+sh)-\varphi(x)}{s}-L_xh}_Y=0.
\]
The operator $L_x$ is unique and it is called $H_0$-Gateaux derivative of $\varphi$ at $x\in X$. We set $\J_{G,H_0} \varphi(x):=L_x$. 
For any $k\in\N$ (in an analogous way of Definition \ref{H-frechet}) we can define the notion of $k$-times $H_0$-Gateaux differentiability of a function $\varphi$ and we denote by $\J_{G,H_0}^k \varphi(x)$ its $k$-order $H_0$-Gateaux derivative at $x\in X$, in particular $\J^k_{G,H_0} \varphi(x)$ belongs to $\mathcal{L}^{(k)}(H_0;Y)$.
We say that $\varphi$ is $k$-times $H_0$-Gateaux differentiable if it is $k$-times $H_0$-Gateaux differentiable at any point $x\in X$.
%
\end{definition}
If $X$ is a Hilbert space and $X=H_0$, then Definitions \ref{H-frechet} and  \ref{H-Gateaux} are the classical notions of Fr\'echet and Gateaux differentiability, respectively, and in this case we will use the notation $\J \varphi$ and $\J_{G} \varphi$, respectively. If $\varphi:X\ra Y$ is $H_0$-Fr\'echet differentiable, then it is $H_0$-Gateaux differentiable and $\J_{H_0}\varphi=\J_{G,H_0}\varphi$; the converse is false. The following result provides a sufficient condition for the equivalence of $H_0$-Fr\'echet and $H_0$-Gateaux differentiability.

\begin{theorem}\label{Gateaux-Frechet}
Let $\varphi:X\ra Y$ be a $H_0$-continuous function. If $\varphi$ is $H_0$-Gateaux differentiable and $\J_{G,H_0}\varphi:X\ra \mathcal{L}(H_0,Y)$ is $H_0$-continuous, then $\varphi$ is $H_0$-Fr\'echet differentiable and $\J_{H_0}\varphi=\J_{G,H_0}\varphi$.
\end{theorem}

\begin{proof}
We refer to \cite[4.1.7. Corollary 1]{Fre80} for the case in which $X$ is a Hilbert space and $X=H_0$. If $H_0\subsetneqq X$, given $x\in X$ let us consider the function $g_x:H_0\ra Y$ defined by
\[
g_x(h):=\varphi(x+h),\qquad h\in H_0.
\]
Since $\varphi$ is $H_0$-continuous,  $g_x:H_0\ra Y$ is continuous. 
By the $H_0$-Gateaux differentiability of $\varphi$, for any $x\in X$ and $h,k\in H_0$, we infer
\begin{align*}
&\lim_{s\ra 0}\norm{\frac{g_x(h+sk)-g_x(h)}{s}-\J_{G,H_0}\varphi(x+h)k}_Y\\
&\qquad\qquad\qquad\qquad=\lim_{s\ra 0}\norm{\frac{\varphi(x+h+sk)-\varphi(x+h)}{s}-\J_{G,H_0}\varphi(x+h)k}_Y=0.
\end{align*}
Thus $g_x$ is Gateaux differentiable at $h\in H_0$ and $\J_Gg_x(h)=\J_{G,H_0}\varphi(x+h)$, for any $x\in X$ and $h\in H_0$. Since $\J_{G,H_0}\varphi:X\ra \mathcal{L}(H_0;Y)$ is $H_0$-continuous by assumption, it follows that $\J_Gg_x:H_0\ra \mathcal{L}(H_0;Y)$ is continuous. By \cite[4.1.7. Corollary 1]{Fre80} we thus infer that $g_x:H_0\ra\R$ is Fr\'echet differentiable at $0$ and $\J_Gg_x(0)=\J g_x(0)$. To conclude, we observe that for any $x\in X$
\begin{align*}
\lim_{\norm{h}_{H_0}\ra 0}&\frac{\|\varphi(x+h)-\varphi(x)-\J g_x(0)h\|_Y}{\norm{h}_{H_0}}=\lim_{\norm{h}_{H_0}\ra 0}\frac{\|g_x(h)-g_x(0)-\J g_x(0)h\|_Y}{\norm{h}_{H_0}}=0,
\end{align*}
so that $\varphi$ is $H_0$-Fr\'echet differentiable at $x\in H$ and $\J_{H_0}\varphi(x)=\J_{G,H_0}\varphi(x)$.
\end{proof}

In the next propositions we collect some basic properties of the $H_0$-Fr\'echet and $H_0$-Gateaux differentiability. The chain rule is particularly valuable, among other tools, especially for establishing gradient estimates for a transition semigroup associated with a stochastic partial differential equations. These estimates are extensively used to investigate both the Schauder and Sobolev regularity of Kolmogorov equations linked to the transition semigroup; please refer to the citations provided in the introduction of this section.

\begin{proposition}
Let $Z$ be a Banach space equipped with the norm $\norm{\cdot}_Z$. If $f:X\ra Y$ is $H_0$-Gateaux differentiable at $x_0\in X$ and $g: Y\rightarrow Z$ is Fr\'echet differentiable at $y_0=f(x_0)$, then $g\circ f$ is $H_0$-Gateaux differentiable at $x_0$ and its $H_0$-Gateaux derivative is $\J g(y_0)\circ \J_{G,H_0}f(x_0)$.
\end{proposition}

\begin{proof}
Let $h\in H_0$ and let  $(t_n)_{n\in\N}$ be an infinitesimal sequence of positive real numbers. We consider the sequence $(z_n)_{n\in\N}\subseteq Z$ defined as
\begin{align*}
z_n:=g(f(x_0+t_nh))-g(f(x_0))-t_n (\J g(y_0)\circ \J_{G,H_0}f(x_0))h.
\end{align*}
We need to prove that $(t_n^{-1}\norm{z_n}_Z)_{n\in\N}$ is an infinitesimal sequence. For $k\in H_0$ and $y\in Y$ set
\begin{align*}
R(k)&:=f(x_0+k)-f(x_0)-\J_{G,H_0}f(x_0)k;\\
S(y)&:=g(y_0+y)-g(y_0)-\J g(y_0)y;
\end{align*}
and 
\begin{align*}
y_n:=\frac{f(x_0+t_n h)-f(x_0)}{t_n}=\J_{G,H_0}f(x_0)h+\frac{R(t_nh)}{t_n}.
\end{align*}
We write 
\begin{align*}
\frac{z_n}{t_n}&=\frac{g(y_0+t_n y_n)-g(y_0)}{t_n}-(\J g(y_0)\circ \J_{G,H_0}f(x_0))h\\
&=\frac{S(t_ny_n)}{t_n}+\J g(y_0)y_n-(\J g(y_0)\circ \J_{G,H_0}f(x_0))h=\frac{S(t_ny_n)}{t_n}+\J g(y_0)\frac{R(t_nh)}{t_n}.
\end{align*}
The $H_0$-Gateaux differentiability of $f$ yields $\lim_{n\ra+\infty}t_n^{-1}\|R(t_nh)\|_Y=0$, whereas 
the Fr\'echet differentiability of $g$ yields $\lim_{n\ra+\infty}t_n^{-1}\|S(t_ny_n)\|_Z=0$. We thus infer $\lim_{n\ra +\infty}t_n^{-1}\norm{z_n}_Z=0$ which concludes the proof.
\end{proof}

\begin{proposition}\label{Lip}
Assume that $\varphi:X\ra Y$ is a $H_0$-Fr\'echet differentiable function and that there exists a constant $M>0$ such that $\|\J_{H_0}\varphi(x)\|_{\mathcal{L}(H_0;Y)}\leq M$, for any $x\in X$. 
The function $\varphi$ is $H_0$-Lipschitz and for every $x\in X$ and $h\in H_0$ it holds
\begin{align*}
\|\varphi(x+h)-\varphi(x)\|_Y\leq M\|h\|_{H_0}.
\end{align*}
\end{proposition}

\begin{proof}
The proof is standard, we give it for the sake of completeness. Let $\phi:[0,1]\ra X$ be defined as $\phi(t):=x+th$ and let $\Psi(t):=\varphi(\phi(t))$, for any $t\in[0,1]$.
Observe that $\Psi$ is derivable in $(0,1)$, indeed for $t\in(0,1)$
\begin{align*}
\Psi'(t)&=\lim_{s\ra 0}\frac{\Psi(t+s)-\Psi(t)}{s}=\lim_{s\ra 0}\frac{\varphi(x+(t+s)h)-\varphi(x+th)}{s}=\mathcal{D}_{H_0}\varphi(x+th)h.
\end{align*}
Furthermore $\Psi$ is continuous in $[0,1]$, since $\varphi$ is $H_0$-Fr\'echet differentiale. By the mean value theorem there exists $t_0\in(0,1)$ such that $\Psi(1)-\Psi(0)=\Psi'(t_0)$. Thus
\begin{align*}
|\varphi(x+h)-\varphi(x)|=|\mathcal{D}_{H_0}\varphi(x+t_0h)h|\leq M\norm{h}_{H_0}.
\end{align*}
This concludes the proof.
\end{proof}

\begin{proposition}
Let $\varphi:X\ra Y$ be a $H_0$-Fr\'echet differentiable function, such that for every $x\in H$ it holds $\|\J_{H_0}\varphi(x)\|_{\mathcal{L}(H_0;Y)}=0$. Then for every $x\in X$ and $h\in H_0$ it holds $\varphi(x+h)=\varphi(x)$. Moreover if $H_0$ is dense in $X$ and $\varphi$ is a continuous function, then $\varphi$ is constant.
\end{proposition}

\begin{proof}
By Proposition \ref{Lip} we get that $\varphi(x+h)=\varphi(x)$, for every $x\in X$ and $h\in H_0$. Now assume that $H_0$ is dense in $X$ and that $\varphi$ is a continuous function. Let $x_0\in X$ and let $(h_n)_{n\in\N}\subseteq H_0$ be a sequence converging to $x_0$ in $X$. By the continuity of $\varphi$ and the first part of the proof of the proposition it holds
\begin{align*}
\varphi(x_0)=\lim_{n\ra+\infty}\varphi(h_n)=\lim_{n\ra+\infty}\varphi(0+h_n)=\varphi(0).
\end{align*} 
This conclude the proof.
\end{proof}


The following result clarifies the relationship between the classical notion of Fr\'echet differentiability and the notion of $H_0$-Fr\'echet differentiability.
\begin{proposition}\label{dalpha}
Let $\varphi:X\ra Y$ be a Fr\'echet differentiable function. $\varphi$ is $H_0$-Fr\'echet differentiable and for any $x\in X$ and $h\in H_0$ it holds $\J_{H_0}\varphi(x)h=\J \varphi(x)h$.
\end{proposition}

\begin{proof}
By the Fr\'echet differentiability of $\varphi$ we know that for every $\eta>0$ there exists $\delta>0$ such that for every $y\in X$ with $0<\|y\|<\delta$ it holds
\begin{align}\label{Shandalar}
\frac{\|\varphi(x+y)-\varphi(x)- \J\varphi(x)y\|_Y}{\|y\|_X}<\eta.
\end{align}
Fix $\eps>0$, let $\eta=\eps/C$ in \eqref{Shandalar}, where $C$ is the constant appearing in \eqref{immercontinua}, and consider $h\in H_0$ such that $0<\|h\|_{H_0}<\delta/C$ where $\delta>0$ is the one introduced at the beginning of the proof. Observe that by \eqref{immercontinua} it holds that $0<\|h\|_X\leq C\|h\|_{H_0}<\delta$. By \eqref{immercontinua} and \eqref{Shandalar}, it holds
\begin{align*}
0\leq \frac{\|\varphi(x+h)-\varphi(x)-\J\varphi(x)h\|_Y}{\|h\|_{H_0}} &=\frac{\|\varphi(x+h)-\varphi(x)-\J\varphi(x)h\|_Y}{\|h\|_X}\frac{\|h\|_X}{\|h\|_{H_0}}\\
&\leq C\frac{\|\varphi(x+h)-\varphi(x)-\J\varphi(x)h\|_Y}{\|h\|_X}<\eps.
\end{align*}
This concludes the proof.
\end{proof}

The converse implication of Proposition \ref{dalpha} is not true in general, as shown by the following example.
\begin{example}
Let $\varphi:X\ra \R$ be defined as
\begin{align*}
\varphi(x):=\eqsys{
\|x\|^2_{H_0}, & x\in H_0;\\
0, &\text{otherwise.}}
\end{align*}
$\varphi$ is not Fr\'echet differentiable (it is not continuous), but it is $H_0$-Fr\'echet differentiable and it holds
\begin{align*}
\J_{H_0}\varphi(x)h=\eqsys{
2\langle x,h\rangle_{H_0}, & x\in H_0;\\
0, & \text{otherwise.}}
\end{align*}
\end{example}

\begin{remark}
One of the most significant frameworks in which the Gross differentiability is applied are abstract Wiener spaces. Let $X$ be a separable Banach space, and let $\gamma$ be a Gaussian measure on the Borel $\sigma$-algebra of $X$. We denote by $H_\gamma$ the Cameron--Martin space associated to $\gamma$ 
(see \cite{BOGIE1,Lunardi}). In this case, we consider Gross differentiability along the Cameron--Martin space $H_\gamma$, namely $H_0 = H_\gamma$ in Definition \ref{H-frechet}. This differentiability is related to the integration by parts formula with respect to $\gamma$ and lays the ground for the theory of infinite-dimensional Ornstein--Uhlenbeck semigroups (see, for example, \cite{BOG18,CMG96,DaPratoZab_Second,DaPratoZab_Stoch,LMP20,LP20}). 
In the most important example of Wiener space $X = C([0,1])$, the space of real-valued continuous functions on $[0,1]$, $\gamma$ is the Wiener measure, and the Cameron--Martin space $H_\gamma$ consists of real-valued functions $f$ defined on $[0,1]$ such that $f$ is absolutely continuous, $f'\in L^2((0,1),d\lambda)$ (here $d\lambda$ is the Lebesgue measure on $(0,1)$), and $f(0)=0$ (see \cite{BOGIE1,Lunardi}).
\end{remark}


\subsubsection{Gross differentiability for real-valued functions}\label{Gross_real}

In this subsection we rewrite Definitions \ref{H-frechet} and \ref{H-Gateaux} for functions from a Hilbert space $H$ (with inner product $\langle\cdot,\cdot,\rangle_H$) with values in $\mathbb{R}$: we will focus of this case from here on. Let $k\in \N$ and $L\in\mathcal{L}^{(k)}(H;\R)$, by the Riesz representation theorem there exists a unique $l\in \mathcal{L}^{(k-1)}(H)$ such that 
\[
L(h_1,\ldots,h_n)=\scal{l(h_1,\ldots,h_{n-1})}{h_n}_H,\qquad h_1,\ldots,h_n\in H.
\]

\begin{definition}\label{gradienti}
Let $k\in\N$ and let $f:H\ra \R$ be a $k$-times
\begin{enumerate}[\rm (i)]
\item $H_0$-Fr\'echet differentiable function, then for any $x\in H$ we denote by $\nabla_{H_0}^kf(x)$ the unique element of $\mathcal{L}^{(k-1)}(H_0)$ such that for any $h_1,\ldots,h_k\in H_0$
\[
\J^k_{H_0} f(x)(h_1,\ldots,h_k)=\langle\nabla^k_{H_0}f(x)(h_1,\ldots,h_{k-1}),h_n \rangle_{H_0}.
\]
If $k=1$ we write $\nabla_{H_0}f(x)$ and we call it $H_0$-gradient of $f$ at $x\in H$. If $H=H_0$ we write $\nabla^kf(x)$.

\item $H_0$-Gateaux differentiable function, then for any $x\in H$ we denote by $\nabla_{G,H_0}^kf(x)$ the unique element of $\mathcal{L}^{k-1}(H_0)$ such that for any $h_1,\ldots,h_k\in H_0$
\[
\J^k_{G,H_0} f(x)(h_1,\ldots,h_k)=\langle\nabla^k_{G,H_0}f(x)(h_1,\ldots,h_{k-1}),h_n\rangle_{H_0}.
\]
If $k=1$ we write $\nabla_{G,H_0}f(x)$ and we call it $H_0$-gradient of $f$ at $x\in H$. If $H=H_0$ we write $\nabla_G^kf(x)$.
\end{enumerate}
\end{definition}

Notice that $\nabla f$ and $\nabla_G f$ are the standard Fr\'echet and Gateaux gradient of $f$ in $x\in H$, respectively.
Now we introduce some natural functional spaces associated to the notion of $H_0$-differentiability.
\begin{definition}\label{spzHr}
Let $k\in\N$. We denote by ${\rm BUC}^k_{H_0}(H)$ the subspace of ${\rm BUC}^k(H)$ of $k$-times $H_0$-Fr\'echet differentiable functions $f:H\ra \R$ such that the functions $x\mapsto\nabla^i_{H_0}f(x)$ belong to ${\rm BUC}(H;\mathcal{L}^{(i-1)}(H_0))$, for every $i=1,\ldots,k$. If $H=H_0$ we write ${\rm BUC}^k(H)$. 
\end{definition}
For any $k \in \mathbb{N}$, the space ${\rm BUC}^k_{H_0}(H)$ is a Banach space if endowed with the norm
\begin{equation*}
\norm{f}_{{\rm BUC}^k_{H_0}(H)}:=\norm{f}_{\infty}+\sum_{i=1}^k\sup_{x\in H}\|\nabla^i_{H_0}f(x)\|_{\mathcal{L}^{(i-1)}(H_0)}.
\end{equation*}
 
We conclude this subsection noting that, for real-valued functions  Theorem \ref{Gateaux-Frechet} reads as follows.
\begin{theorem}\label{Gateaux-FrechetGrad}
Let $\varphi:H\ra \R$ be a $H_0$-continuous function. If $\varphi$ is $H_0$-Gateaux differentiable and $\nabla_{G,H_0}\varphi:H\ra H$ is $H_0$-continuous, then $\varphi$ is $H_0$-Fr\'echet differentiable and $\nabla_{H_0}\varphi=\nabla_{G,H_0}\varphi$.
\end{theorem}

\subsection{Differentiability in the sense of Cannarsa and Da Prato}\label{BeppeDif}
We introduce here the notion of $R$-differentiability considered by P. Cannarsa and G. Da Prato in \cite{CDP96} dropping the assumption of injectivity of the operator $R$ considered in that paper.
In a separable Hilbert space $H$ (with inner product $\langle\cdot,\cdot,\rangle_H$ and associated norm $\norm{\cdot}_H$), we fix a self-adjoint operator $R\in\mathcal{L}(H)$. We denote by $\ker R$ the kernel of $R$ and by $(\ker R)^{\bot}$ its orthogonal subspace in $H$. By $P_{\ker R}$ we denote the orthogonal projection on $\ker R$.

We denote by $H_R:=R(H)$ the range of the operator $R$. In order to provide $H_R$ with a Hilbert structure, we recall that the restriction $R_{|_{(\ker R)^{\bot}}}:(\ker R)^{\bot}\rightarrow H_R$ is a bijective operator. Hence, we can define the pseudo-inverse of $R$ as 
\begin{align}\label{Chiamo}
R^{-1}:=(R_{|_{(\ker R)^{\bot}}})^{-1} \in \mathcal{L}(H_R,(\ker R)^{\perp}),
\end{align} 
see \cite[Appendix C]{LI-RO1}.
We introduce the scalar product 
\begin{equation}\label{Rprod}
\scal{x}{y}_{H_R}:=\langle R^{-1}x,R^{-1}y\rangle_H,\quad x,y\in H_R
\end{equation}
with its associated norm $\norm{x}_{H_R}:=\|R^{-1}x\|_H$. Endowed with this inner product $H_R$ is a separable Hilbert space and a Borel subset of $H$ (see \cite[Theorem 15.1]{KE1}). A possible orthonormal basis of $H_R$ is given by $\{Re_k\}_{k\in\N}$, where $\{e_k\}_{k\in\N}$ is an orthonormal basis of $(\ker R)^{\bot}$. We recall that it holds
\begin{align}
RR^{-1} &=\Id_{H_R},&\!\!\!\!\!\!\!\!\!\!\!\!\!\!\!\!\!\!\!\!\!\!\!\!\!\!\!\!\!\!\!\!\!\!\!\!\!\! R^{-1}R&=\Id_{H}-P_{\ker R},\label{orietta1}\\ 
R^{-1}R_{|_{(\ker R)^{\perp}}}&=\Id_{(\ker R)^{\perp}},&\!\!\!\!\!\!\!\!\!\!\!\!\!\!\!\!\!\!\!\!\!\!\!\!\!\!\!\!\!\!\!\!\!\!\!\!\!\! RR^{-1}&=\Id_H.\label{orietta2}
\end{align}
Notice that for any $x\in H_R$ it holds
\[
\norm{x}_H=\|RR^{-1}x\|_H\leq \norm{R}_{\mathcal{L}(H)}\|R^{-1}x\|_{H}\leq \norm{R}_{\mathcal{L}(H)}\norm{x}_{H_R}.
\]
Thus, when $H_0=H_R$ the constant $C$ appearing in \eqref{immercontinua} is given by $\norm{R}_{\mathcal{L}(H)}$. 
Moreover we recall that $\ker R=\{0\}$ if, and only if, $R(H)$ is dense in $H$ (see \cite[Lemma VI.2.8]{DS88I}).

%

\begin{definition}\label{C-D_diff}
We say that a function $f:H\ra\R$ is $R$-differentiable at $x\in H$ if there exists $l_x\in H$ such that for any $v\in H$ it holds 
\begin{equation}\label{C-D_1}
\lim_{s\ra 0}\abs{\frac{f(x+sRv)-f(x)}{s}-\langle l_x,v\rangle_H}=0.
\end{equation}
We set $\nabla_Rf(x):=l_x$. We say that a function is $R$-differentiable if it is $R$-differentiable at any $x\in H$.
We say that $f$ is twice $R$-differentiable at $x\in H$ if it is $R$-differentiable and there exists a unique $B_x\in \mathcal{L}(H)$ such that for any $v\in H$ we have
\begin{equation}\label{C-D_2}
\lim_{s\ra 0}\norm{\frac{\nabla_R f(x+sRv)-\nabla_Rf(x)}{s}-B_xv}_{H}=0.
\end{equation}
We set $\nabla^2_R f(x):=B_x$. Let $k\in\N$; similarly one introduces the notion of $k$-times $R$-dif\-fer\-en\-tia\-bil\-i\-ty at $x\in H$. We denote by $\nabla^k_R f(x)\in\mathcal{L}^{(k-1)}(H)$ the $k$-order $R$-derivative of $\varphi$.
We say that a function is $k$-times $R$-differentiable when it is $k$-times $R$-differentiable at any $x\in H$.
\end{definition}

\begin{remark}
In \cite{CDP96} the authors introduce a weaker notion of twice $R$-differentiability. More precisely, a function $\varphi:H\ra\R$ is twice $R$-differentiable if, for any $x\in H$, there exists a unique $B_x\in \mathcal{L}(H)$ such that for any $w,v\in H$ it holds 
\[
\lim_{s\ra 0}\scal{\frac{\nabla_R\varphi(x+sRv)-\nabla_R\varphi(x)}{s}-B_xv}{w}_H=0.
\]
\end{remark}

We introduce some natural functional spaces associated to the notion of $R$-differentiability.

\begin{definition}\label{spazC-D}
For any $k\in\N$, we denote by ${\rm BUC}^k_{R}(H)$ the subspace of ${\rm BUC}^k(H)$ of $k$-times $R$-Fr\'echet differentiable functions $\varphi:H\ra \R$ such that the mapping $x\mapsto\nabla^i_R\varphi(x)$ belongs to ${\rm BUC}(H;\mathcal{L}^{(i-1)}(H))$, for every $i=1,\ldots,k$. %
%
\end{definition}
The space  ${\rm BUC}^k_{R}(H)$ equipped with the norm
\[
\norm{\varphi}_{{\rm BUC}^k_{R}(H)}:=\norm{\varphi}_{\infty}+\sum_{i=1}^k\sup_{x\in H}\|\nabla_{R}^i\varphi(x)\|_{\mathcal{L}^{(i-1)}(H)}
\]
is a Banach space. The following result will be useful throughout the paper. Notice that if $\ker R=\{0\}$ (as in \cite{CDP96}) the next proposition is trivial.

\begin{proposition}\label{ortogonaleR}
Let $k\in\N$ and let $f:H\ra\R$ be a $k$-times $R$-differentiable function. For any $x\in H$
\begin{align}\label{O1}
\nabla^k_Rf(x)\in \mathcal{L}^{(k-1)}(H;(\ker R)^{\bot}),
\end{align}
where we set $\mathcal{L}^{(0)}(H;(\ker R)^{\bot}):=(\ker R)^{\bot}$. In other words, for any $v\in \ker R$
\begin{align*}
\scal{\nabla_Rf(x)}{v}_H=0,
\end{align*}
and if $k\geq 2$ 
\begin{align}\label{O2}
\nabla^k_Rf(x)(v_1,\ldots,v_{k-1})=0,
\end{align}
whether $v_i\in \ker R $ for $i=1,\ldots,k-1$.
\end{proposition}

\begin{proof}
The case $k\leq 2$ is an immediate consequence of \eqref{C-D_1} and \eqref{C-D_2} by taking $v\in \ker R$. For $k>2$ we proceed by induction. Assume the assertion to hold true for $k$ and let us prove it for $k+1$.
Let $f:H\ra\R$ be a $(k+1)$-times $R$-differentiable function. 
By the inductive hypothesis and Definition \ref{C-D_diff} we infer that
\begin{align*}
\lim_{s\ra 0}\norm{\frac{\nabla^k_R f(x+sRv_k)-\nabla^k_Rf(x)}{s}- \nabla^{k+1}_Rf(x)(\cdot,\ldots,\cdot,v_{k})}_{\mathcal{L}^{(k)}(H;(\ker R)^{\bot})}=0,
\end{align*}
hence \eqref{O1} holds true and \eqref{O2} holds true when $v_n\in\ker R$. Moreover for any $v_1,\ldots,v_{n}\in H$ we have
\begin{align*}
\lim_{s\ra 0}&\bigg\|\frac{\nabla^k_R f(x+sRv_k)(v_1,\ldots,v_{k-1})-\nabla^k_Rf(x)(v_1,\ldots,v_{k-1})}{s}-\nabla^{k+1}_Rf(x)(v_1,\ldots,v_{k})\bigg\|_{H}=0,
\end{align*}
thus the  inductive hypothesis yields \eqref{O2}. 
\end{proof}

\subsection{Comparisons between $R$-differentiability and $H_R$-dif\-fer\-en\-tiability}\label{RvsH}

We aim to compare the notion of $R$-differentiability of Section \ref{BeppeDif} with the notion of $H_0$-differentiability of Section \ref{Gross_real}, if $H_0=H_R$.




\begin{proposition}\label{ugua}
A function $\varphi:H\ra\R$ is $R$-differentiable if and only if it is $H_R$-Gateaux differentiable. Moreover, for any $x\in H$ 
\begin{align}
\scal{R\nabla_R\varphi(x)}{h}_{H_R}&=\scal{\nabla_{G,H_R}\varphi(x)}{h}_{H_R}, &h\in H_R\notag\\
\scal{\nabla_{R}\varphi(x)}{v}_{H}&=\langle R^{-1}\nabla_{G,H_R}\varphi(x),v\rangle_{H}, &v\in H\label{uguaH}.
\end{align}
In particular, for any $x\in H$ it holds $\|\nabla_R\varphi(x)\|_H=\|\nabla_{G,H_R}\varphi(x)\|_{H_R}$.

\end{proposition}

\begin{proof}
Assume that $\varphi$ is $R$-differentiable. By \eqref{Rprod}, Definition \ref{C-D_diff} and Proposition \ref{ortogonaleR}, for every $x\in H$, $h\in H_R$ of the form $h=Rv$ it holds
\begin{align*}
\lim_{s\ra 0}\bigg|\frac{\varphi(x+sh)-\varphi(x)}{s}&-\langle R\nabla_R\varphi(x),h\rangle_{H_R}\bigg|=\lim_{s\ra 0}\abs{\frac{\varphi(x+sRv)-\varphi(x)}{s}-\langle \nabla_R\varphi(x),v\rangle_{H}}=0.
\end{align*}
Since $R\nabla_R \varphi(x)\in H_R$, the mappings $h\mapsto\langle R\nabla_R\varphi(x),h\rangle_{H_R}$ belongs to $H_R^*$, so $\varphi$ is $H_R$-Gateaux differentiable and 
\(
\scal{\nabla_{G,H_R}\varphi(x)}{h}_{H_R}=\langle R\nabla_R\varphi(x),h\rangle_{H_R}.
\)
Assume now that $\varphi$ is $H_R$-Gateaux differentiable. Recalling that $h=Rv$, by Definitions \ref{H-Gateaux} and \ref{gradienti}, and \eqref{orietta1} for every $x\in H$, $v\in H$ it holds
\begin{align*}
&\lim_{s\ra 0}\abs{\frac{\varphi(x+sRv)-\varphi(x)}{s}-\langle R^{-1}\nabla_{G,H_R}\varphi(x),v\rangle_{H}}\\
&=\lim_{s\ra 0}\abs{\frac{\varphi(x+sh)-\varphi(x)}{s}-\langle \nabla_{G,H_R}\varphi(x),h\rangle_{H_R}-\langle R^{-1}\nabla_{G,H_R}\varphi(x),P_{\ker R}v\rangle_{H}}.
\end{align*}
Since, by \eqref{Chiamo}, $R^{-1}\nabla_{G,H_R}\varphi(x)\in (\ker R)^{\bot}$ and $h=Rv$, by \eqref{Rprod} we obtain
\begin{align*}
\lim_{s\ra 0}\abs{\frac{\varphi(x+sRv)-\varphi(x)}{s}-\langle R^{-1}\nabla_{G,H_R}\varphi(x),v\rangle_{H}}=0.
\end{align*}
Hence  $\varphi$ is $R$-differentiable and \eqref{uguaH} is verified. 
\end{proof}

Bearing in mind Definitions \ref{spzHr} and \ref{spazC-D}, we now show that ${\rm BUC}^k_{H_R}(H)={\rm BUC}^k_R(H)$ for any $k\in\N$. We need the following preliminary result.

%

\begin{lemma}\label{iso}
For any $n\in\N$ 
the mapping $T_{n}:\mathcal{L}^{(n)}((\ker R)^{\bot})\ra \mathcal{L}^{(n)}(H_R)$ defined for $v_1,\ldots,v_{n}\in H_R$ and $A\in \mathcal{L}^{(n)}((\ker R)^{\bot})$ as
\[
(T_{n}A)(v_1,\ldots,v_{n}):=RA(R^{-1}v_1,\ldots,R^{-1}v_n),
\]
is a linear isometry and an isomorphism. We recall that for $n=0$ we let $\mathcal{L}^{(0)}((\ker R)^{\bot}):=(\ker R)^{\bot}$ and $\mathcal{L}^{(0)}(H_R):=H_R$ and we set 
\(T_{0}v:=Rv,\) 
for any $v\in (\ker R)^{\bot}$.
Furthermore if $A\in \mathcal{L}^{(n)}((\ker R)^{\bot})$ and $v\in H_R$ it holds
\begin{align*}
T_{n-1}(A(\cdot,\ldots,\cdot,R^{-1}v))=(T_{n}A)(\cdot,\ldots,\cdot,v).
\end{align*}
\end{lemma}

\begin{proof}
For any $n\in\N$, since $R_{|_{(\ker R)^{\perp}}}:(\ker R)^{\perp}\rightarrow R(H)$ is linear and bijective, it follows that $T_{n}$ is linear. By \eqref{orietta1} and \eqref{orietta2}, for any $A\in \mathcal{L}^{(n)}((\ker R)^\perp)$ we have 
\begin{align*}
\norm{T_{n}A}_{\mathcal{L}^{(n)}(H_R)}&=\sup_{v_1,\ldots,v_n\in H_R\backslash\{0\}}\frac{\|RA(R^{-1}v_1,\ldots,R^{-1}v_n)\|_{H_R}}{\|v_1\|_{H_R}\cdots\|v_n\|_{H_R}}\\
&=\sup_{v_1,\ldots,v_n\in H_R\backslash\{0\}}\frac{\|A(R^{-1}v_1,\ldots,R^{-1}v_n)\|_H}{\|R^{-1}v_1\|_H\cdots\|R^{-1}v_n\|_H}\\
&=\sup_{h_1,\ldots,h_n\in (\ker R)^{\bot}\backslash \{0\}}\frac{\|A(h_1,\ldots,h_n)\|_H}{\|h_1\|_H\cdots\|h_n\|_H}=\norm{A}_{\mathcal{L}^{(n)}((\ker R)^{\bot})}.\qedhere
\end{align*}
\end{proof}

\begin{theorem}\label{identificazionen}
For any $n\in\N$, it holds ${\rm BUC}^n_{H_R}(H)={\rm BUC}^n_{R}(H)$.
Moreover if $\varphi\in {\rm BUC}^n_{R}(H)$ and $x\in H$ then
\begin{align}\label{idem}
\nabla^{n}_{H_R}\varphi(x)= T_{n-1}\left(\nabla^{n}_{R}\varphi(x)\right),
\end{align}
with $T_{n-1}$ as in Lemma \ref{iso}.
\end{theorem}

\begin{proof}
We proceed by induction. We start by proving the base case $n=1$.
Let $\varphi:H\ra \R$; by Proposition \ref{ugua} the mapping $x\mapsto\nabla_{G,H_R}\varphi(x)$ belongs to ${\rm BUC}(H;H_R)$ if, and only if, the mapping $x\mapsto\nabla_{R}\varphi(x)$ belongs to ${\rm BUC}(H;H)$. Thus the case $n=1$ follows by Theorem \ref{Gateaux-FrechetGrad}.


Now we prove the induction step. Assume the thesis to be true for an integer $n\geq 2$. Let $\varphi\in {\rm BUC}^{n+1}_R(H)$, $x\in H$ and $v_n\in H_R\backslash\{0\}$ such that $v_n=Rh_n$ with $h_n\in (\ker R)^{\bot}$. By Proposition \ref{ortogonaleR} and Lemma \ref{iso} we infer
\begin{align*}
&\lim_{s\ra 0}\norm{\frac{\nabla^{n}_{H_R}\varphi(x+sv_{n})-\nabla^{n}_{H_R}\varphi(x)}{s}-T_{n-1}\left(\nabla^{n+1}_{R}\varphi(x)(\cdot,\ldots,\cdot,h_n)\right)}_{\mathcal{L}^{(n-1)}(H_R)}\notag\\
&\qquad=\lim_{s\ra 0}\norm{T_{n-1}\left(\frac{\nabla^{n}_{R}\varphi(x+sv_{n})-\nabla^{n}_{R}\varphi(x)}{s}-\nabla^{n+1}_{R}\varphi(x)(\cdot,\ldots,\cdot,h_n)\right)}_{\mathcal{L}^{(n-1)}(H_R)}\notag\\
&\qquad=\lim_{s\ra 0}\norm{\frac{\nabla^{n}_{R}\varphi(x+sRh_{n})-\nabla^{n}_{R}\varphi(x)}{s}-\nabla^{n+1}_{R}\varphi(x)(\cdot,\ldots,\cdot,h_n)}_{\mathcal{L}^{(n-1)}((\ker R)^{\bot})}=0.
\end{align*}
Since 
\[
T_{n-1}\left(\nabla^{n+1}_{R}\varphi(x)(\cdot,\ldots,\cdot,R^{-1}v_n)\right)=(T_{n}\nabla^{n+1}_{R}\varphi(x))(\cdot,\ldots,\cdot,v_n),
\]
we obtain \eqref{idem} and ${\rm BUC}^n_{R}(H)\subseteq {\rm BUC}^n_{H_R}(H)$ . The inclusion ${\rm BUC}^n_{H_R}(H)\subseteq {\rm BUC}^n_{R}(H)$ follows in a similar way using the operator $T_n^{-1}$ instead of the operator $T_n$.
\end{proof}

In view of the above result, from here on we will use the space ${\rm BUC}^k_R(H)$ to represent both ${\rm BUC}^k_{H_R}(H)$ and ${\rm BUC}^k_R(H)$.


\subsection{A Comparison with the classical notions of differentiability}\label{Real}
We focus here on the relationship between the $R$-differentiability and $H_R$-dif\-fer\-en\-tiability, and the classical Fr\'echet and Gateaux differentiability.

\begin{proposition}\label{gat-Rdiff}
For any $n\in\N$, if $\varphi:H\ra\R$ is $n$-times Gateaux differentiable, then $\varphi$ is $n$-times $R$-differentiable and for any $x\in H$ and $n\geq 2$ it holds
\[
\nabla^n_R\varphi(x)(v_1,\ldots,v_{n-1})=R\nabla^n_G\varphi(x)(Rv_1,\ldots,Rv_{n-1}),\qquad v_1,\ldots,v_{n-1}\in H.
\]
While if $n=1$, then for any $x\in H$
\[\nabla_R\varphi(x)=R\nabla_G\varphi(x).\]
\end{proposition}

\begin{proof}
We proceed by induction.
Let $\varphi:H\ra\R$ be a Gateaux differentiable function and let $x,v\in H$. By Definition \ref{gradienti} we have
\begin{equation*}
\lim_{s\ra 0}\abs{\frac{\varphi(x+sRv)-\varphi(x)}{s}-\langle \nabla_G\varphi(x),Rv\rangle_H}=0.
\end{equation*}
Thus the thesis follows for $n=1$.  Now we assume that the statements hold true for $n$ and we prove it for $n+1$. Let $\varphi:H\ra\R$ be a $(n+1)$-times Gateaux differentiable function and let $x,v_1,\ldots,v_n\in H$. By the inductive hypothesis we have for any $s\in\R\setminus\{0\}$ 
\begin{align}
&\bigg\|\frac{\nabla^n_R\varphi(x+sRv_n)(v_1,\ldots,v_{n-1})-\nabla^n_R\varphi(x)(v_1,\ldots,v_{n-1})}{s}-R\nabla^{n+1}_G\varphi(x)(Rv_1,\ldots,Rv_n)\bigg\|_{H}\notag\\
&=\bigg\|\frac{R\nabla_G^n\varphi(x+sRv_n)(Rv_1,\ldots,Rv_{n-1})-R\nabla^n_G\varphi(x)(Rv_1,\ldots,Rv_{n-1})}{s}\notag\\
&\qquad\qquad\qquad\qquad\qquad\qquad\qquad\qquad\qquad\qquad\qquad-R\nabla^{n+1}_G\varphi(x)(Rv_1,\ldots,Rv_n)\bigg\|_{H}\notag\\
&\leq\norm{R}_{\mathcal{L}(H)}\bigg\|\frac{\nabla^n_G\varphi(x+sRv_n)(Rv_1,\ldots,Rv_{n-1})-\nabla^n_G\varphi(x)(Rv_1,\ldots,Rv_{n-1})}{s}\notag\\
&\qquad\qquad\qquad\qquad\qquad\qquad\qquad\qquad\qquad\qquad\qquad-\nabla^{n+1}_G\varphi(x)(Rv_1,\ldots,Rv_n)\bigg\|_{H}.\label{Brando}
\end{align}
To conclude it is enough to take the limit as $s$ approaches zero in \eqref{Brando}.
\end{proof}

Combining Theorem \ref{Gateaux-FrechetGrad}, Propositions \ref{identificazionen} and \ref{gat-Rdiff} we obtain the following result.

\begin{theorem}\label{equiGDC-GF}
For any $k\in\N$, if $\varphi\in {\rm BUC}^k(H)$ then $\varphi\in {\rm BUC}_{R}^k(H)$ (and so it belongs in ${\rm BUC}_{H_R}^k(H)$, by Proposition \ref{identificazionen}) and for any $x\in H$ and $k\geq 2$ it holds
\begin{align*}
\nabla^k_{H_R}\varphi(x)(h_1,\ldots,h_{k-1})&=R^2\nabla^k\varphi(x)(h_1,\ldots,h_{k-1}),& h_1,\ldots,h_{k-1}\in H_R;\\
\nabla^k_R\varphi(x)(v_1,\ldots,v_{k-1})&=R\nabla^k\varphi(x)(Rv_1,\ldots,Rv_{k-1}),& v_1,\ldots,v_{k-1}\in H.
\end{align*}
Furthermore if $k=1$, for any $x\in H$
\begin{align}
\label{relation_Malliavin}
\nabla_{H_R}\varphi(x)=R^2\nabla\varphi(x),\qquad\text{ and }\qquad
\nabla_R\varphi(x)=R\nabla\varphi(x).
\end{align}
\end{theorem}

\section{Malliavin calculus in Wiener spaces}
\label{Mal_top_sec}

We start by considering a Gaussian framework. We introduce on $(H,\mathcal{B}(H))$ a centered (that is with zero mean) Gaussian measure $\gamma$ with covariance operator $Q$. Here $Q\in\mathcal{L}(H)$ is a self-adjoint non-negative and trace class operator. 
The aim of this Section is to recall the construction of the Malliavin derivative operators in the sense of Gross and in the sense of Cannarsa--Da Prato (mainly referring to the books \cite{BOGIE1} and \cite{DaPrato}, respectively); then to show that they can be interpreted as two (different) examples of the general notion of Malliavin derivative (see  Appendix \ref{Malliavin_abstract_sec}).
In particular, we will show that the Malliavin derivative in the sense of Gross and in the sense of Cannarsa--Da Prato are different operators but with the same domain.

\subsection{The Gaussian Hilbert space $H_\gamma^*$}
\label{X^*_g}
We will denote by $(H^*)'$ the algebraic dual of $H^*$, namely the space of all linear (not necessarely continuous) functional  $f:H^*\ra\R$. The space $H^*$ is included in $L^2(H,\mathcal{B}(H), \gamma)$ and the inclusion mapping $j:H^* \rightarrow L^2(H, \gamma)$ is continuous. The space
\begin{equation*}
H_\gamma^*:= \text{closure of $j(H^*)$ in $L^2(H, \gamma)$},
\end{equation*}
when endowed with the scalar product of $L^2(H,\gamma)$, is a Gaussian Hilbert space (see e.g. \cite[Lemma 2.2.8]{BOGIE1}). We introduce the covariance operator $R_\gamma: H_\gamma^* \rightarrow (H^*)'$ defined as
\begin{equation*}
R_\gamma f(g):=\langle f, j(g)\rangle_{L^2(H, \gamma)}=\int_Hfj(g)d\gamma, \qquad f \in H^*_\gamma, \ g \in H^*.
\end{equation*}
$R_\gamma$ is injective and its range is contained in $H$ (see e.g. \cite[Proposition 2.3.6]{Lunardi}).

We define the Cameron--Martin space $K$ (for the measure $\gamma$) as $K:=R_\gamma(H_\gamma^*)\subseteq H$. $K$ inherits a structure of separable Hilbert space through $R_\gamma$ (see e.g. \cite[Lemma 2.4.1]{BOGIE1}), that is introducing the mapping
\[
\hat\cdot:=R_\gamma^{-1}: K \rightarrow H^*_\gamma\subseteq L^2(H,\gamma)
\]
it holds 
\begin{equation*}
\langle h,k\rangle _K:=\langle \hat h, \hat k\rangle_{L^2(H, \gamma)}=\int_H\hat{h}\hat{k}d\gamma,
\end{equation*}
whenever $h, k \in K$ with $h=R_\gamma \hat h$, $k=R_\gamma \hat k$.
As proved in \cite[Theorem 4.2.7]{Lunardi}, the Cameron--Martin space coincide with the Hilbert space $H_{Q^{1/2}}=Q^{1/2}(H)$ and its inner product is given by
\begin{equation*}
\scal{h}{k}_K=\scal{h}{k}_{Q^{1/2}}:=\langle Q^{-1/2}h,Q^{-1/2}k\rangle_{H},\qquad h,k\in K=H_{Q^{1/2}}.
\end{equation*}
From the very definition of the Cameron--Martin space it follows that the mapping $\hat\cdot:=R_\gamma^{-1}$ is a unitary operator and this yields that
\begin{equation}
\label{star}
H_\gamma^*=\{\hat{h}\in L^2(H,\gamma)\,|\, h \in K\},
\end{equation}
where every $\hat h\in H_\gamma^*$ is a centered Gaussian random variable with variance $\|\hat h\|^2_{L^2(H, \gamma)}=\|h\|^2_K$.

On the other hand, when the measure $\gamma$ is non degenerate (that is $\ker Q=\{0\}$), the Cameron--Martin space turns out to be dense in $H$ (see e.g. \cite[Lemma 2.16]{DaPrato}). In this case, see \cite[Section 2.5.2]{DaPrato}, the mapping $\hat \cdot =R_\gamma^{-1}:K \rightarrow H_\gamma^*\subseteq L^2(H, \gamma)$ can be uniquely extended to a linear isometry $\mathcal{W}_\bullet$ defined as 
\[\mathcal{W}_{\bullet}:H \rightarrow H^*_\gamma \subseteq L^2(H, \gamma).\] 
In the literature the mapping $\mathcal{W}_{\bullet}$ is usually called white noise mapping. Thus,  $\mathcal{W}_{\bullet}$ is a unitary operator and it holds
\begin{equation}
\label{star_bis}
H_\gamma^*=\{\mathcal{W}_z\in L^2(H,\gamma)\,|\, z \in H\}.
\end{equation}
Every  $\mathcal{W}_z$ is a centered Gaussian random variable with variance $\|\mathcal{W}_z\|^2_{L^2(H, \gamma)}=\|z\|^2_H$.

\subsection{Sobolev spaces}
\label{gradient_Sobolev}
We denote by $\nabla_{H_{Q^{1/2}}}$ and $\nabla_{Q^{1/2}}$ the gradient operators introduced in Definitions \ref{gradienti} and \ref{C-D_diff}, respectively, with the choice $R=Q^{1/2}$ and $H_0=H_{Q^{1/2}}$. In Section \ref{sec_2} we analyzed the relations between this two operators.
 
\begin{lemma}\label{lemma_equiv}
Let $Q\in\mathcal{L}(H)$ be a self-adjoint non-negative and trace class operator with $\ker Q=\{0\}$. For any $\varphi \in {\rm BUC}^1(H)$,
\begin{equation*}
\langle\nabla_{H_{Q^{1/2}}}\varphi(x), h \rangle_{H_{Q^{1/2}}}= \langle \nabla_{Q^{1/2}} \varphi(x), z\rangle_H, \qquad z \in H, \ h \in H_{Q^{1/2}} \ \text{with} \ h=Q^{1/2}z.
\end{equation*}
In particular,
\(
\|\nabla_{H_{Q^{1/2}}}\varphi(x)\|_{H_{Q^{1/2}}} =\|\nabla_{Q^{1/2}}  \varphi(x)\|_H. 
\)
\end{lemma}   
The following integration by parts formula with respect to $\gamma$ is well known (see e.g. \cite[Theorem 5.1.8]{BOGIE1})
\begin{equation}
\label{i.b.p.Banach}
\int_H \langle\nabla_{H_{Q^{1/2}}} \varphi, h\rangle_{H_{Q^{1/2}}} d\gamma= \int_H \varphi \hat h  d\gamma, \qquad \varphi\in C^1_{b}(H),\ h \in H_{Q^{1/2}},
\end{equation}
and in \cite[Chapter 5]{BOGIE1} it is used to prove that the operator
\begin{equation*}
\nabla_{H_{Q^{1/2}}}:C_b^1(H)\subseteq L^p(H,\gamma)\ra L^p(H,\gamma; {H_{Q^{1/2}}}),
\end{equation*}
is closable as an operator from $ L^p(H,\gamma)$ to $L^p(H,\gamma; {H_{Q^{1/2}}})$, for any $p \in[1,+\infty)$; for a proof see \cite[Proposition 9.3.7]{Lunardi}.
The Sobolev spaces $W_{{H_{Q^{1/2}}}}^{1,p}(H,\gamma)$ are defined as the domain of the closure of the operator $\nabla_{H_{Q^{1/2}}}$, still denoted by $\nabla_{H_{Q^{1/2}}}$, in $L^p(H,\gamma)$. $W_{{H_{Q^{1/2}}}}^{1,p}(H,\gamma)$ is a Banach space with the norm
\begin{equation}
\label{norm_D_H}
\|f\|_{W_{{H_{Q^{1/2}}}}^{1,p}(H,\gamma)}^p:=\|f\|^p_{L^p(H, \gamma)}+\|\nabla_{H_{Q^{1/2}}}f\|^p_{L^p(H,\gamma; {H_{Q^{1/2}}})}.
\end{equation}
Moreover, the integration by parts formula \eqref{i.b.p.Banach} holds for any $\varphi$ belonging to $W_{H_{Q^{1/2}}}^{1,p}(H,\gamma)$ and $h \in H_{Q^{1/2}}$ (see e.g. \cite[Proposition 9.3.10]{Lunardi}).

When $\gamma$ is non degenerate, Lemma \ref{lemma_equiv} provides the following equivalent form of the integration by parts formula \eqref{i.b.p.Banach} 
\begin{equation}
\label{ibpnew}
\int_H \langle \nabla_{Q^{1/2}} \varphi, z\rangle_H d\gamma = \int_H \varphi  \mathcal{W}_zd\gamma, \qquad \varphi \in C_b^1(H), \ z \in H,
\end{equation}
where we used $\hat h=\mathcal{W}_z$ for $h=Q^{1/2} z$. The integration by parts formula \eqref{ibpnew} is the one used in \cite{DaPrato}
to prove that the operator
$\nabla_{Q^{1/2}}:C_b^1(H) \rightarrow L^p(H,\gamma;H)$ is closable as an unbounded operator from $L^p(H, \gamma)$ to $L^p(H, \gamma;H)$, for any $ p\in[1,+\infty)$. The Sobolev space $W_{Q^{1/2}}^{1,p}(H,\gamma)$ is defined as the domain of the closure of the operator $\nabla_{Q^{1/2} }$, denoted by $M$. It is a Banach space with the norm
\begin{equation}
\label{norm_M}
\|f\|_{W_{Q^{1/2}}^{1,p}(H,\gamma)}^p:=\|f\|^p_{L^p(H, \gamma)}+\|M f\|^p_{L^p(H, \gamma; H)}.
\end{equation}
Moreover, the integration by parts formula \eqref{ibpnew} holds for any $\varphi \in W_{Q^{1/2}}^{1,p}(H,\gamma)$ and $z \in H$.

In the following sections we show that the gradient operators $\nabla_{H_{Q^{1/2}}}$ and $M$, can be thought as Malliavin derivative operators. For this purpose, referring back to Section \ref{Malliavin_abstract_sec}, it will be enough to identify the choices of the probability space $(\Omega, \mathcal{F}, \mathbb{P})$, the Gaussian Hilbert space $\mathcal{H}_1$, the Hilbert space $\mathcal{H}$ and the unitary operator $W$.

\begin{remark}
Given $R\nabla:C^1_b(H)\subseteq L^p(H, \gamma) \rightarrow L^p(H, \gamma;H)$, under specific compatibility assumptions between $R$ and $Q$ it is possible to prove that $R\nabla$ is closable, we call generalized gradient the closure of it (see, for example, \cite{BF22,GGvN03}). The Sobolev space $W^{1,p}_R(H, \gamma)$ is the domain of the closure of the operator $R\nabla$. 
See also \cite{Add-Mur-Ros2022} for the problem of equivalence of Sobolev norms.

\end{remark}

\subsection{Malliavin derivative in the sense of Gross}
\label{Bogachev}
  
In  \cite{GRO1} (see also \cite{BOGIE1}), the reference probability space is $(\Omega, \mathcal{F}, \mathbb{P})=(H, \mathcal{B}(H), \gamma)$, with $\gamma$ a centered Gaussian measure. The Gaussian Hilbert space $\mathcal{H}_1$ is  $H_\gamma^*$, the space $\mathcal{H}$ is the Cameron--Martin space $K=H_{Q^{1/2}}$ and  the unitary operator $W$ is the operator $\hat \cdot=R_\gamma^{-1}$.  With these identifications, by comparing the integration by parts formula 
 \begin{equation*}
\int_H \langle\nabla_{H_{Q^{1/2}}} \varphi, h\rangle_{H_{Q^{1/2}}} d\gamma= \int_H \varphi \hat  h d\gamma,  \qquad \varphi \in W_{{H_{Q^{1/2}}}}^{1,2}(H,\gamma)={\rm Dom}(\nabla_{H_{Q^{1/2}}}), \  h \in H_{Q^{1/2}}
\end{equation*}
with \eqref{ibp_1}:
\begin{equation*}
\mathbb{E}\left[ \langle D\varphi, h \rangle_{\mathcal{H}}\right]=\mathbb{E}\left [\varphi W(h)\right], \qquad \forall \varphi \in \mathbb{D}^{1,2}={\rm Dom}(D), \ h \in \mathcal{H},
\end{equation*}
we immediately see that the Malliavin derivative in \cite{GRO1} (see also \cite{BOGIE1}) is the gradient operator $\nabla_{H_{Q^{1/2}}}$.
 
\subsection{Malliavin derivative in the sense of Cannarsa and Da Prato}
 \label{DaPrato}
In \cite{DaPrato} the reference probability space is $(\Omega, \mathcal{F}, \mathbb{P})=(H, \mathcal{B}(H), \gamma)$, with $\gamma$ a centered non degenerate Gaussian measure. The Gaussian Hilbert space $\mathcal{H}_1$ is $H_\gamma^*$, the space $\mathcal{H}$ is $H$ itself and  the unitary operator $W$ is the white noise mapping $\mathcal{W}_\bullet$. 
With these identifications, by comparing the integration by parts formula 
\begin{equation*}
\int_H \langle M \varphi, z\rangle_H d\gamma = \int_H \varphi \mathcal{W}_z d\gamma,\qquad \varphi \in W_{Q^{1/2}}^{1,2}(H,\gamma)={\rm Dom}(M), \  z \in H,
\end{equation*}
with \eqref{ibp_1}, we immediately see that the Malliavin derivative in \cite{DaPrato} is the gradient operator $M$.

\subsection{Final remarks}
\label{fin_rem}
 
The Malliavin derivatives $\nabla_{H_{Q^{1/2}}}$ and $M$ of Sections \ref{Bogachev} and \ref{DaPrato} are different. Indeed \eqref{relation_Malliavin} yields the relation $\nabla_{H_{Q^{1/2}}}=Q^{1/2}M$. On the other hand, the domain of the two derivatives is the same, that is 
\[
W_{Q^{1/2}}^{1,2}(H,\gamma)=W_{H_{Q^{1/2}}}^{1,2}(H,\gamma).
\] 
In fact, thanks to Lemma \ref{lemma_equiv}, the closure of the space $C_b^1(H)$ with respect to the norm \eqref{norm_D_H} is the same as its closure with respect to the norm \eqref{norm_M}.
This should not be surprising in light of the general results of Section \ref{Malliavin_abstract_sec}: in Sections \ref{Bogachev} and \ref{DaPrato} the reference Gaussian Hilbert space $\mathcal{H}_1$ is the same, that is $H_\gamma^*$; thus Proposition \ref{charD12} ensures the two Malliavin derivatives $\nabla_{H_{Q^{1/2}}}$ and $M$ to have the same domain.
What changes in Sections \ref{Bogachev} and \ref{DaPrato} is how the space $\mathcal{H}_1$ is characterized.  In Section \ref{Bogachev} we considered the unitary operator $\hat \cdot =R_\gamma^{-1}$ between $H_{Q^{1/2}}$ and $\mathcal{H}_1$ and obtain the characterization \eqref{star}, whereas in Section \ref{DaPrato} we considered the unitary operator $\mathcal{W}_\bullet$ between $H$ and $\mathcal{H}_1$ and obtain the characterization \eqref{star_bis}. This naturally leads to different Malliavin derivatives, having chosen different Hilbert spaces $\mathcal{H}$ and unitary operators $W$.

\section{Application: Lasry--Lions approximation and an interpolation result}\label{Lasry}

We consider the same framework of Section \ref{sec_2}.
We introduce here the notions of $H_0$-H\"older and $R$-H\"older functions. We prove this notions to be equivalent when $H_0=H_R$. We thus prove an interpolation type result for the space of $H_R$-H\"older functions. A key role in the proof is played by Lasry--Lions type approximations along the space $H_R$ (see Subsection \ref{Mourinho}).

\subsection{H\"older and Lipschitz functions along subspaces}

We recall here the notions of $H_0$-H\"olderianity ($H_0$-Lipschitzianity, respectively) and $R$-H\"olderianity ($R$-Lipschitzianity, respectively) and show that they are equivalent when $H_0=H_R$.

\begin{definition}
We say that $\varphi:H\ra \R$ is a $H_0$-H\"older function of exponent $\alpha\in(0,1)$ if there exists a positive constant $L_{\alpha,H_0}$ such that for any $x\in H$ and $h\in H_0$ it holds
\begin{align}\label{H-holder}
|\varphi(x+h)-\varphi(x)|\leq L_{\alpha,H_0}\norm{h}^\alpha_{H_0}.
\end{align}
The infimum of all the possible constants $L_{\alpha,H_0}$ appearing in \eqref{H-holder} is called $H_0$-H\"older constant constant of $\varphi$. 
\end{definition}
It is trivial to see that a $H_0$-H\"older function $\varphi:H\ra \R$ is $H_0$-continuous. When $H_0=H$ we recover the classical definition of H\"older continuous function from $H$ to $\R$.
Moreover, by \eqref{immercontinua}, if $\varphi$ is H\"older continuous, then $\varphi$ is $H_0$-H\"older. The converse is not true as shown by the following example. 
\begin{example}
For any $\alpha\in (0,1)$, we consider the function $\varphi_\alpha:H\ra \R$ defined as
\begin{align*}
\varphi_\alpha(x):=\eqsys{
\|x\|^{\alpha}_{H_0}, & x\in H_0;\\
0, & \text{otherwise.}}
\end{align*}
$\varphi_\alpha$ is $H_0$-H\"older of exponent $\alpha$, but it is not continuous. 
\end{example}
%
%

\begin{definition}
For any $\alpha\in(0,1)$ we denote by ${\rm BUC}^\alpha_{H_0}(H)$ the subspace of ${\rm BUC}(H)$ given by all $H_0$-H\"older functions of exponent $\alpha$. 
\end{definition}

\noindent For any $\alpha\in (0,1)$, the space ${\rm BUC}^\alpha_{H_0}(H)$ is a Banach space, if endowed with the norm
\begin{align*}
||\varphi||_{{\rm BUC}^\alpha_{H_0}(H)}:=\norm{\varphi}_\infty+[\varphi]_{H_0,\alpha},
\end{align*}
where
\begin{align*}
[\varphi]_{H_0,\alpha}:=\sup_{\substack{x\in H;\\ h\in H_0\setminus\{0\}}}\frac{|\varphi(x+h)-\varphi(x)|}{\|h\|^\alpha_{H_0}}.
\end{align*}
If $H=H_0$ we write ${\rm BUC}^\alpha(H)$ and  $[\varphi]_{\alpha}$.

\begin{definition}
Let $\alpha\in(0,1)$. We say that $\varphi:H\ra \R$ is $R$-H\"older of exponent $\alpha$ if there exists $L_{\alpha,R}>0$ such that for any $x,v\in H$ it holds 
\begin{align}\label{R-holder}
|\varphi(x+Rv)-\varphi(x)|\leq L_{\alpha,R}\norm{v}^\alpha_{H}.
\end{align}
The infimum of all the possible constants $L_{\alpha,R}$ appearing in \eqref{R-holder} is called $R$-H\"older constant constant of $\varphi$. 
\end{definition}

\begin{definition}
Let $\alpha\in(0,1)$. We denote by ${\rm BUC}^\alpha_R(H)$ the subspace of ${\rm BUC}(H)$ of the $R$-H\"older functions of exponent $\alpha$.
\end{definition}

\noindent For any $\alpha\in (0,1)$, the space ${\rm BUC}^\alpha_{R}(H)$ is a Banach space, 
if endowed with the norm
\[
\norm{f}_{R,\alpha}:=\norm{f}_{\infty}+[f]_{R,\alpha},
\]
where
\[
[f]_{R,\alpha}:=\sup_{x,v\in H,\; v\neq 0}\dfrac{\abs{f(x+Rv)-f(x)}}{\norm{v}^\alpha_H}.
\]

Let us compare the above definitions in the specific case $H_0=H_R$.

\begin{proposition}
If $\alpha\in (0,1)$, then ${\rm BUC}^{\alpha}_{H_R}(H)={\rm BUC}^{\alpha}_{R}(H)$. 
\end{proposition}
\begin{proof}
Simply letting $h=Rv$, it immediately follows that \eqref{R-holder} coincides with \eqref{H-holder}.
\end{proof}
In view of the above result, from here on we will use the space ${\rm BUC}_R^\alpha(H)$ to represent both ${\rm BUC}^\alpha_{H_R}(H)$ and ${\rm BUC}_R^\alpha(H)$. We state now a useful characterization of the space ${\rm BUC}_R^\alpha(H)$ whenever $\ker R=\{0\}$.

\begin{proposition}
Assume that $\ker R=\{0\}$ and let $\alpha\in (0,1)$ and $\varphi\in {\rm BUC}(H)$. $\varphi$ belongs to ${\rm BUC}_{R}^{\alpha}(H)$ if, and only if, the function $\varphi\circ R$ belongs to ${\rm BUC}^{\alpha}(H)$. Furthermore it holds
\begin{equation*}
[\varphi]_{R,\alpha}=[\varphi\circ R]_{\alpha}.
\end{equation*}
\end{proposition}

\begin{proof} 
Let us start by noticing that $H_R$ is dense in $H$, since $\ker R=\{0\}$. We begin to prove that $\varphi\in {\rm BUC}_{R}^{\alpha}(H)$ implies $\varphi\circ R\in {\rm BUC}^{\alpha}(H)$.
If $\varphi\in{\rm BUC}_{R}^\alpha(H)$, then for any $x,y\in H$ it holds
\begin{align*}
|(\varphi\circ R)(x)-(\varphi\circ R)(y)| &=|\varphi(Rx)-\varphi(Ry)|\\
&=|\varphi(Ry+(Rx-Ry))-\varphi(Ry)|\leq[\varphi]_{R,\alpha}\|x-y\|_{H}^\alpha.
\end{align*}
So $\varphi\circ R\in{\rm BUC}^\alpha(H)$ and $[\varphi\circ R]_{\alpha}\leq [\varphi]_{R,\alpha}$. 

Now let $\varphi\circ R\in{\rm BUC}^\alpha(H)$, $x\in H$ and let $(x_n=Ry_n)_{n\in\N}\subseteq H_R$ be a sequence converging to $x$ in $H$. For any $v\in H$ it follows  
\begin{align*}
|\varphi(x+Rv)-\varphi(x)|&
=\lim_{n\ra+\infty}|\varphi(Ry_n+Rv)-\varphi(Ry_n)|\\
&=\lim_{n\ra+\infty}|(\varphi\circ R)(y_n+v)-(\varphi\circ R)(y_n)|\leq [\varphi\circ R]_\alpha\|v\|_H^\alpha.
\end{align*}
So $\varphi\in{\rm BUC}_{R}^\alpha(H)$ and $[\varphi]_{R,\alpha}\leq [\varphi\circ R]_{\alpha}$.
\end{proof}

Now we introduce the notion of $R$-Lipschitz function.

\begin{definition}
We say that $\varphi:H\ra \R$ is $R$-Lipschitz, respectively if there exists $L_{R}>0$ such that for any $x,v\in H$ it holds 
\begin{align}\label{R-lip}
|\varphi(x+Rv)-\varphi(x)|\leq L_{R}\norm{v}_{H}.
\end{align}
The infimum of all the possible constants $L_{R}$ appearing in \eqref{R-lip} is called $R$-Lipschitz constant constant of $\varphi$. 
\end{definition}

\begin{definition}
We denote by ${\rm Lip}_{b,R}(H)$ the subspace of ${\rm BUC}(H)$ of the $R$-Lipschitz function.
\end{definition}
\noindent ${\rm Lip}_{b,R}(H)$ is a Banach space, if endowed with the norm
\[
\norm{f}_{{\rm Lip}_{b,R}(H)}:=\norm{f}_{\infty}+[f]_{R},
\]
where
\[
[f]_{R}:=\sup_{x,v\in H,\; v\neq 0}\frac{\abs{f(x+Rv)-f(x)}}{\norm{v}_H}.
\]
It easy to see that \eqref{R-lip} is equivalent to \eqref{H-Lip}. Hence by Proposition \ref{Lip} and Theorem \ref{identificazionen}, we deduce that 
\(
{\rm BUC}^1_R(H)\subseteq {\rm Lip}_{b,R}(H).
\)

\subsection{Lasry--Lions type approximations}\label{Mourinho}
We recall the classical Lasry--Lions approximating procedure introduced in  \cite{LL86}. 
\begin{theorem}\label{Lions}
Let $f\in  {\rm BUC}(H)$ and $t>0$; we define the function
\begin{align*}
S(t)f(x):=\sup_{z\in H}\set{\inf_{y\in H}\set{f(x+z-y)+\frac{1}{2t}\|y\|_H^2}-\frac{1}{t}\|z\|_H^2},\qquad x\in H.
\end{align*}
Then $\{S(t)f\}_{t\geq 0}\subseteq{\rm BUC}^1(H)$ and for any $x\in H$ it holds
\(
\lim_{t\ra 0^+}\abs{S(t)f(x)-f(x)}=0.
\)
\end{theorem}

We now recall a modification of the Lasry--Lions approximating procedure presented in \cite{CDP96} (if $\ker R=\{0\}$): given $f\in {\rm BUC}(H)$ and $t>0$ one defines the function 
\begin{align}\label{LLCD0}
\mathcal{S}^R(t)f(x):=\sup_{w\in H}\set{\inf_{v\in H}\set{f(v)+\frac{1}{2t}\|R^{-1}(w-v)\|_H^2}-\frac{1}{t}\|R^{-1}(w-x)\|_H^2},\qquad x\in H
\end{align}
with the convention that $\|R^{-1}y\|=+\infty$ if $y\notin R(H)$.
We will consider a slight modification of \eqref{LLCD0} obtained via a change of variables 
\begin{align}\label{LLCD}
S^R(t)f(x):=\sup_{h\in H_R}\set{\inf_{k\in H_R}\set{f(x+k-h)+\frac{1}{2t}\|k\|_{H_R}^2}-\frac{1}{t}\|h\|_{H_R}^2},\qquad x\in H.
\end{align}



\begin{proposition}\label{Lio}
For every $f\in {\rm BUC}(H)$ and $t>0$, the mapping $x\mapsto S^R(t)f(x)$ belongs to ${\rm BUC}(H)$.
\end{proposition}

\begin{proof}
%
Fix $t>0$. We prove that $S^R(t)f\in{\rm BUC}(H)$.  Since $f$ is uniformly continuous we know that for every $\eta>0$ there exists $\delta:=\delta(\eta)>0$ such that for every $x,y\in H$ with $0<|x-y|<\delta$ it holds $|f(x)-f(y)|<\eta$. Let $x,y\in H$ be such that $0<|x-y|<\delta$, then for every $\sigma>0$ there exist $h_\sigma,k_\sigma\in H_R$ such that
\begin{align*}
S^R(t)f(x)-S^R(t)f(y)&\leq \inf_{k\in H_R}\set{f(x+h_\sigma-k)+\frac{1}{2\eps}\|k\|_R^2}-\frac{1}{\eps}\|h_\sigma\|_R^2+\sigma\\
&\qquad\qquad\qquad-\inf_{k\in H_R}\set{f(y+h_\sigma-k)+\frac{1}{2\eps}\|k\|_R^2}+\frac{1}{\eps}\|h_\sigma\|_R^2\\
&\leq f(x+h_\sigma-k_\sigma)+\frac{1}{2\eps}\|k_\sigma\|_R^2-f(y+h_\sigma-k_\sigma)-\frac{1}{2\eps}\|k_\sigma\|_R^2+2\sigma\\
&\leq \eta+2\sigma.
\end{align*}
Using similar arguments we get that $S^R(t)f(x)-S^R(t)f(y)\geq -\eta-2\sigma$. So $S^R(t)f$ is uniformly continuous.
\end{proof} 


The following proposition summarize some of the properties of $\{S^R(t)f\}_{t\geq0}$ that we will use throughout this section.

\begin{proposition}
Let $f\in {\rm BUC}^\alpha_R(H)$, for some $\alpha\in(0,1)$. Let $\{S^R(t)f\}_{t\geq 0}$ be the family of functions introduced in \eqref{LLCD}.
There exists $c_\alpha>0$ such that for every $t>0$ and $x\in H$ it holds
\begin{align}
\|S^R(t)f\|_\infty&\leq\|f\|_\infty;\label{LL_limitatezza}\\
0\leq f(x)-S^R(t)f(x)&\leq c_\alpha [f]_{R,\alpha}^{2/(2-\alpha)}t^{\alpha/(2-\alpha)};\label{LL_approximation}\\
[S^R(t)f]_R &\leq 2\big(2c_\alpha [f]_{R,\alpha}^{2/(2-\alpha)}\big)^{1/2}t^{(\alpha-1)/(2-\alpha)}.\label{LL_derivative_estimate}
\end{align}
In particular the mapping $x\mapsto S^R(t)f(x)$ belongs to ${\rm Lip}_b(H)$, for every $t>0$.
\end{proposition}

\begin{proof}
We start by proving \eqref{LL_limitatezza}. 
\begin{align}
S^R(t)f(x)&=\sup_{h\in H_R}\set{\inf_{k\in H_R}\set{f(x+h-k)+\frac{1}{2t}\|k\|_{H_R}^2}-\frac{1}{t}\|h\|_{H_R}^2}\notag\\
&\leq \sup_{h\in H_R}\set{f(x)+\frac{1}{2t}\|h\|_{H_R}^2-\frac{1}{t}\|h\|_{H_R}^2}\leq f(x)\leq \|f\|_\infty.\label{Gerald1}
\end{align}
In a similar way 
\begin{align}
S^R(t)f(x)&=\sup_{h\in H_R}\set{\inf_{k\in H_R}\set{f(x+h-k)+\frac{1}{2t}\|k\|_{H_R}^2}-\frac{1}{t}\|h\|_{H_R}^2}\notag\\
&\geq \inf_{k\in H_R}\set{f(x-k)+\frac{1}{2t}\|k\|_{H_R}^2}\geq -\|f\|_\infty.\label{Gerald2}
\end{align}
By \eqref{Gerald1} and \eqref{Gerald2} we get \eqref{LL_limitatezza}.

Let us now prove \eqref{LL_approximation}. By \eqref{LLCD}, for every $\eta>0$ there exists $k_\eta\in H_R$ such that
\begin{align}
0\leq f(x)-S^R(t)f(x)&\leq f(x)-\inf_{k\in H_R}\set{f(x-k)+\frac{1}{2t}\|k\|_{H_R}^2}\notag\\
&\leq f(x)-f(x-k_\eta)-\frac{1}{2t}\|k_\eta\|_{H_R}^2+\eta\notag\\
&\leq [f]_{R,\alpha}\|k_\eta\|_{H_R}^\alpha-\frac{1}{2t}\|k_\eta\|_{H_R}^2+\eta.\label{Regan}
\end{align}
From the above inequality we get the estimate
\(
\|k_\eta\|_{H_R}^2\leq 2t [f]_{R,\alpha}\|k_\eta\|_{H_R}^\alpha+2t \eta.
\)
The Young inequality yields, for every $c>0$,
\begin{align*}
\|k_\eta\|_{H_R}^2\leq \frac{\alpha}{2}c^{2/\alpha} \|k_\eta\|_{H_R}^2+\frac{2-\alpha}{2}\frac{1}{c^{2/(2-\alpha)}}  (2t [f]_{R,\alpha})^{2/(2-\alpha)}+2t \eta.
\end{align*}
Now taking $c=\alpha^{-\alpha/2}$ we get 
\begin{align}\label{Truman}
\|k_\eta\|_{H_R}^2\leq (2-\alpha)\alpha^{\alpha/(2-\alpha)}2^{2/(2-\alpha)}[f]_{R,\alpha}^{2/(2-\alpha)}t^{2/(2-\alpha)}+4t\eta.
\end{align}
Combining \eqref{Regan} and \eqref{Truman} we obtain
\begin{align*}
0\leq f(x)-S^R(t)f(x)&\leq [f]_{R,\alpha}\big((2-\alpha)\alpha^{\alpha/(2-\alpha)}2^{2/(2-\alpha)}[f]_{R,\alpha}^{2/(2-\alpha)}t^{2/(2-\alpha)}+4t\eta\big)^{\alpha/2}+\eta.
\end{align*}
Since the above estimate holds for every $\eta>0$, by choosing $\eta$ arbitrarily small, we get \eqref{LL_approximation}.

We conclude by proving \eqref{LL_derivative_estimate}. First notice that by \eqref{LLCD} for every $\sigma>0$ there exists $h_\sigma\in H_R$ such that
\begin{align*}
S^R(t)f(x)\leq \inf_{k\in H_R}\set{f(x+h_\sigma-k)+\frac{1}{2t}\|k\|_{H_R}^2}-\frac{1}{t}\|h_\sigma\|_{H_R}^2+\sigma.
\end{align*}
A straightforward calculation gives 
\begin{align*}
\frac{1}{t}\|h_\sigma\|_{H_R}^2\leq f(x)-S^R(t)f(x)+\sigma+\frac{1}{2t}\|h_\sigma\|_{H_R}^2.
\end{align*}
Thus from \eqref{LL_approximation} we obtain
\begin{align}\label{Discipline}
\|h_\sigma\|_{H_R}^2\leq 2 c_\alpha[f]_{R,\alpha}^{2/(2-\alpha)}t^{2/(2-\alpha)}+2t\sigma.
\end{align}
By \eqref{Discipline} we get
\begin{align*}
S^R(t)f(x+h)-S^R(t)f(x)&\leq \inf_{k\in H_R}\set{f(x+h+h_\sigma-k)+\frac{1}{2t}\|k\|_{H_R}^2}-\frac{1}{t}\|h_\sigma\|_{H_R}^2+\sigma\\
&\qquad -\inf_{k\in H_R}\set{f(x+h+h_\sigma-k)+\frac{1}{2t}\|k\|_{H_R}^2}+\frac{1}{t}\|h+h_\sigma\|_{H_R}^2\\
&=\frac{1}{t}\|h+h_\sigma\|_{H_R}^2-\frac{1}{t}\|h_\sigma\|_{H_R}^2+\sigma=\frac{1}{t}\|h\|_{H_R}^2+\frac{2}{t}\langle h,h_\sigma\rangle_{H_R}+\sigma\\
&\leq \frac{1}{t}\|h\|_{H_R}^2+\frac{2}{t}\|h\|_{H_R}(2 c_\alpha[f]_{R,\alpha}^{2/(2-\alpha)}t^{2/(2-\alpha)}+2t\sigma)^{1/2}+\sigma.
\end{align*}
Since the above inequalities hold for every $\sigma>0$ taking the infimum we get
\begin{align*}
S^R(t)f(x+h)-S^R(t)f(x)\leq \frac{1}{t}\|h\|_{H_R}^2+2\|h\|_{H_R}(2 c_\alpha[f]_{R,\alpha}^{2/(2-\alpha)})^{1/2}t^{(\alpha-1)/(2-\alpha)}.
\end{align*}
In a similar way we get
\begin{align*}
S^R(t)f(x+h)-S^R(t)f(x)\geq -\frac{1}{t}\|h\|_{H_R}^2-2\|h\|_{H_R}(2 c_\alpha[f]_{R,\alpha}^{2/(2-\alpha)})^{1/2}t^{(\alpha-1)/(2-\alpha)}.
\end{align*}
and so 
\begin{align}\label{aaaaaaa}
\abs{S^R(t)f(x+h)-S^R(t)f(x)}\leq \frac{1}{t}\|h\|_{H_R}^2+2\|h\|_{H_R}(2 c_\alpha[f]_{R,\alpha}^{2/(2-\alpha)})^{1/2}t^{(\alpha-1)/(2-\alpha)}.
\end{align}
By \eqref{aaaaaaa}, the mapping $x\mapsto S^R(t)f(x)$ verifies \eqref{R-lip} for every $h\in H_R$ such that $\norm{h}_{H_R}\leq 1$, instead, since $S^R(t)f$, if $\norm{h}_{H_R}> 1$ then
\[
|S^R(t)f(x+h)-S^R(t)f(x)|\leq 2\|S^R(t)f\|_\infty\leq 2\|S^R(t)f\|_\infty\norm{h}_{H_R},
\]
so the proof is concluded.
\end{proof}

\subsection{An interpolation result}
We have now all the ingredients to prove an interpolation result for the space ${\rm BUC}^{\alpha}_{R}(H)$. 
We shall use the $K$ method for real interpolation spaces (see \cite{Lun18,Tri95}). Let $\K_1$ and $\K_2$ be two Banach spaces, with norms $\norm{\cdot}_{\K_1}$ and $\norm{\cdot}_{\K_2}$, respectively. If $\K_2\subseteq \K_1$ with a continuous embedding, then for every $r>0$ and $x\in \K_1$ we define
\begin{align}
\label{K}
K(r,x):=\inf\set{\|a\|_{\K_1}+r\|b\|_{\K_2}\tc x=a+b,\ a\in \K_1,\ b\in \K_2}.
\end{align}
For any $\vartheta\in(0,1)$, we set
\begin{align}
\|x\|_{(\K_1,\K_2)_{\vartheta,\infty}}&:=\sup_{r>0}r^{-\vartheta}K(t,x);\label{defn_norm_interp}\\
(\K_1,\K_2)_{\vartheta,\infty}&:=\{x\in \K_1\,|\, \|x\|_{(\K_1,\K_2)_{\vartheta,\infty}}<+\infty\}.\notag
\end{align}
It is standard to show that $(\K_1,\K_2)_{\vartheta,\infty}$ endowed with the norm $\norm{\cdot}_{(\K_1,\K_2)_{\vartheta,\infty}}$ is a Banach space.  The following result can be found in \cite{CDP96-2} for the case $R={\rm Id}_{H}$ and a similar result can be found in \cite{BFS}, where the space ${\rm Lip}_{b,R}(H)$ is substituted by another space.

\begin{theorem}\label{intDS}
Let $\alpha\in(0,1)$. Up to an equivalent renorming, it holds 
\[
{\rm BUC}_R^\alpha(H)
=
({\rm BUC}(H),{\rm Lip}_{b,R}(H))_{\alpha,\infty}.
\]
\end{theorem}

\begin{proof}
We start by showing that $({\rm BUC}(H),{\rm Lip}_{b,R}(H))_{\alpha,\infty}\subseteq {\rm BUC}_R^\alpha(H)$.
For any element $\varphi\in ({\rm BUC}(H),{\rm Lip}_{b,R}(H))_{\alpha,\infty}$ and any $r,t>0$ there exist $f_{r,t}\in{\rm BUC}(H)$ and $g_{r,t}\in {\rm Lip}_{b,R}(H)$ such that 
\begin{align*}
\varphi(x)=f_{r,t}(x)+g_{r,t}(x),\qquad x\in H;
\end{align*}
and
\begin{align}\label{ceffo}
\|f_{r,t}\|_\infty+r\|g_{r,t}\|_{{\rm Lip}_{b,R}(H)}\leq r^\alpha\|\varphi\|_{({\rm BUC}(H),{\rm Lip}_{b,R}(H))_{\alpha,\infty}}+t.
\end{align}
By \eqref{ceffo}, for any $x,v\in H$ it holds
\begin{align*}
|\varphi &(x+Rv)-\varphi(x)|\leq 2\|f_{r,t}\|_\infty+|g_{r,t}(x+Rv)-g_{r,t}(x)|\leq 2\|f_{r,t}\|_\infty+[g_{r,t}]_R\|v\|_{H}\\
&\leq  2 r^\alpha \|\varphi\|_{({\rm BUC}(H),{\rm Lip}_{b,R}(H))_{\alpha,\infty}}+2t+r^{\alpha-1}\|\varphi\|_{({\rm BUC}(H),{\rm Lip}_{b,R}(H))_{\alpha,\infty}}\|v\|_{H}+\frac{t}{r}\|v\|_{H}.
\end{align*}
Now letting $t$ tend to zero and setting $r=\|v\|_{H}$ we get
\begin{align*}
|\varphi(x+h)-\varphi(x)|&\leq 3\|\varphi\|_{({\rm BUC}(H),{\rm Lip}_{b,R}(H))_{\alpha,\infty}}\|v\|_{H}^\alpha.
\end{align*}
This proves the continuous embedding $({\rm BUC}(H),{\rm Lip}_{b,R}(H))_{\alpha,\infty}\subseteq {\rm BUC}_R^\alpha(H)$.

To show that ${\rm BUC}_R^\alpha(H)\subseteq ({\rm BUC}(H),{\rm Lip}_{b,R}(H))_{\alpha,\infty}$, let $\varphi\in{\rm BUC}_R^\alpha(H)$. For every $t>0$ let $S^R(t)\varphi$ be the function defined in \eqref{LLCD}. For $r\in (0,1)$ we consider the functions $f_r:H\ra\R$ and $g_r:H\ra\R$ defined by
\begin{align*}
f_r(x):=\varphi(x)-S^R(r^{2-\alpha})\varphi(x),\qquad g_r(x):=S^R(r^{2-\alpha})\varphi(x),
\end{align*}
so that $\varphi=f_r+ g_r$ with $f_r \in{\rm BUC}_R(H)$ and $g_r \in {\rm Lip}_{b,R}(H)$ in virtue of Proposition \ref{Lio}.
By \eqref{LL_approximation} we get that there exists a constant $k_1=k_1(\alpha,\varphi)>0$ such that 
$\|f_r\|_\infty\leq k_1 r^\alpha$.
By \eqref{LL_limitatezza} and \eqref{LL_derivative_estimate}, there exist a constant $k_2=k_2(\alpha,\varphi)>0$ such that
\begin{align*}
\|g_r\|_{{\rm Lip}_{b,R}(H)}&=\|S^R(r^{2-\alpha})\varphi\|_\infty+[ S^R(r^{2-\alpha})\varphi]_R\leq k_2r^{\alpha-1}.
\end{align*}
Thus, bearing in mind \eqref{K}, for every $r\in(0,1)$ we get $K(r,\varphi)\leq (k_1+k_2)r^{\alpha}$.
Notice that the previous estimate is trivial if $r>1$. Keeping in mind \eqref{defn_norm_interp} we get the thesis.
\end{proof}
\begin{remark}\label{dege}
In the case $\ker R=\{0\}$ the results of Subsection \ref{Mourinho} were already proved in \cite{BFS} and \cite{CDP96}. Here we proved that the condition $\ker R=\{0\}$ is not necessary to ensure that the Lasry--Lions approximants defined in \eqref{LLCD} have sufficient regularity to prove the interpolation result stated in Theorem \ref{intDS}
\end{remark}


%
%
%
%

\appendix
\section{Malliavin calculus in an abstract framework}
\label{Malliavin_abstract_sec}
Malliavin calculus is named after P. Malliavin who first introduced this tool with his seminal work \cite{Mal78} (see also \cite{Malliavin}).
There he laid the foundations of what is now known as the ``Malliavin calculus'', an infinite-dimensional differential calculus in a Gaussian framework, and used it to give a probabilistic proof  of H\"ormander theorem. This new calculus proved to be extremely successful and soon a number of authors studied variants and simplifications, see e.g. \cite{Bis1,Bis2,GT,KS84,KS85,KS87,Nualart,Shi1,Shi2,Stroock,W84,Zakai}.

The general context consists of a probability space $(\Omega, \mathcal{F}, \mathbb{P})$ and a Gaussian separable Hilbert space $\mathcal{H}_1$, that is a closed subspace of $L^2(\Omega,\mathcal{F}, \mathbb{P})$ consisting of centered Gaussian random variables. The space $\mathcal{H}_1$ (also known as the first Wiener Chaos) induces an orthogonal decomposition, known as the Wiener Chaos Decomposition, of the corresponding $L^2(\Omega,\sigma(\mathcal{H}_1),\mathbb{P})$ space of square integrable random variables that are measurable with respect to the $\sigma$-field generated by $\mathcal{H}_1$. To characterize elements in $\mathcal{H}_1$ it is useful to fix a separable Hilbert space $\mathcal{H}$ and consider a unitary operator between the two spaces.
In this abstract setting one can introduce the notion of Malliavin derivative, that is the derivative $D\varphi$ of a square integrable random variable $\varphi: \Omega \rightarrow \mathbb{R}$, measurable with respect to $\sigma(\mathcal{H}_1)$. Heuristically one differentiates $\varphi$ with respect to $\omega \in \Omega$. 

Usually $\Omega$ is a linear topological space and the Malliavin derivative operator can be introduced as a differential operator (see Section \ref{Mal_top_sec}). Nevertheless, as done for instance in \cite{Nualart}, it is possible to introduce a notion of Malliavin derivative without assuming any topological or linear structure on the probability space $\Omega$. This approach proves to be particularly flexible and useful in several applications; 
moreover, it is general enough to admit as special cases the definitions of Malliavin derivative given in probability spaces with a linear topological structure, as explained in details in Section \ref{Mal_top_sec}.
It is worth mentioning that, in quantum probability theory, there are connections with Malliavin calculus as well. For example, in the general framework of Fock spaces, the so-called annihilation operator can be interpreted as a Malliavin derivative, as discussed in \cite{Mey2006}. Moreover, for a definition of the Malliavin derivative on non-commutative spaces, we refer to \cite{Fra1}.

We point out here that it would be more accurate to speak of \textit{a} (choice of) Malliavin derivative rather than \textit{the} Malliavin derivative. In fact, given $(\Omega, \mathcal{F}, \mathbb{P})$ and the Gaussian Hilbert spaces $\mathcal{H}_1$,
one can construct infinitely many different Malliavin derivative operators. 
On the other hand, it turns out that all these Malliavin derivatives have the same domain when the Gaussian Hilbert space $\mathcal{H}_1$ is the same. This is showed in details in a concrete situation in Section \ref{Mal_top_sec}: there we provide two different (among the infinitely many) examples of Malliavin derivatives on a Wiener space having the same domain. 

In this Section we briefly recall the construction of the Malliavin derivative in the abstract framework described above and collect some results. We mainly refer to \cite{Janson,NP,Nualart,TZ}.

\subsection{Gaussian Hilbert spaces}

Let $(\Omega, \mathcal{F}, \mathbb{P})$ be a probability space, we denote by $\mathbb{E}$ the expectation under $\mathbb{P}$. Let $\mathcal{H}$ be a real separable Hilbert space with inner product $\langle \cdot, \cdot\rangle _{\mathcal{H}}$ and corresponding norm $\norm{\cdot}_{\mathcal{H}}$.

\begin{definition}
\label{Gaussian_space} 
A Gaussian linear space is a real linear space of random variables, defined on $(\Omega, \mathcal{F}, \mathbb{P})$, such that each variable in the space is centered and Gaussian. 
A Gaussian Hilbert space is a Gaussian linear space which is complete, i.e. a closed subspace of $L^2(\Omega,\mathcal{F}, \mathbb{P})$ consisting of centered Gaussian random variables. We denote it by $\mathcal{H}_1$.
\end{definition}

We recall that a linear isometry between Hilbert spaces is a linear map that preserves the inner product. Linear isometries that are onto are called unitary operators.

\begin{proposition}
\label{Gaussian_isometries}
Let $\mathcal{H}$ be a 
Hilbert space. There exists a Gaussian Hilbert space $\mathcal{H}_1$ (with the same dimension of $\mathcal{H}$) and a unitary operator $h \mapsto W(h)$ of $\mathcal{H}$ onto  $\mathcal{H}_1$. That is, $\mathcal{H}_1=\{W(h)\,|\, h \in \mathcal{H}\}$ and  for any $h, k \in \mathcal{H}$,
\begin{equation*}
\mathbb{E}\left[W(h)W(k) \right]= \langle h, k \rangle_{\mathcal{H}}.
 \end{equation*}
\end{proposition}
\begin{proof}
Let $\{e_i\}_{i \in I}$ be an orthonormal basis of $\mathcal{H}$. 
Let $\{\xi_i\}_{i \in I}$ be a collection of independent standard Gaussian random variables, defined on some probability space $(\Omega, \mathcal{F},\mathbb{P})$
. Every element $h \in \mathcal{H}$ can be uniquely written as $h=\sum_{i \in I} \langle h, e_i\rangle_{\mathcal{H}}e_i$. We introduce the mapping $\mathcal{H} \ni h \mapsto W(h):= \sum_{i \in I}\langle h, e_i\rangle_{\mathcal{H}} \xi_i$. 
By construction, the random variable $W(h)$ is Gaussian. Moreover, since the $\xi_i$ are independent, centered and have unit variance, $W(h)$ is centered and, for any $h, k \in \mathcal{H}$, it holds
\begin{equation*}\mathbb{E}\left[W(h)W(k) \right]=\mathbb{E} \left[ \sum_{i \in I}\langle h, e_i\rangle_{\mathcal{H}} \xi_i \sum_{j \in I}\langle k, e_j\rangle_{\mathcal{H}} \xi_j \right]=\sum_{i \in I}\langle h, e_i\rangle_{\mathcal{H}}\langle k, e_i\rangle_{\mathcal{H}}= \langle h, k \rangle_{\mathcal{H}}.
\end{equation*}
This entails that $W$ is a unitary operator of $\mathcal{H}$ onto the Gaussian Hilbert space $\mathcal{H}_1:=\{W(h) \,|\, h \in \mathcal{H}\}$, and concludes the proof.
 \end{proof}


In \cite{Nualart} the unitary operator $W$, introduced in Proposition \ref{Gaussian_isometries}, is called isonormal Gaussian processes.
The role of the space $\mathcal{H}$ and the operator $W$, in the above result, is to suitable index the elements in $\mathcal{H}_1$. We point out that, fixed a generic Gaussian Hilbert space $\mathcal{H}_1$, there are infinitely many possible choices of real Hilbert spaces $\mathcal{H}$ (with the same dimension as $\mathcal{H}_1$) and unitary operators $W$ such that $\mathcal{H}_1=\{W(h) \,|\, h \in \mathcal{H}\}$. For instance, since $\mathcal{H}_1$ is itself a real Hilbert space (with respect to the usual $L^2(\Omega,\mathcal{F},\mathbb{P})$ inner product), it follows that $\mathcal{H}_1$ can be represented by choosing $\mathcal{H}$ equal to $\mathcal{H}_1$ itself and $W$ equal to the identity operator. In general, given an Hilbert space $\mathcal{H}$, there are infinitely many different ways of choosing an orthonormal basis $\{e_i\}_{i \in I}$ in $\mathcal{H}$ and an orthonormal basis $\{\xi_i\}_{i \in I}$ in $\mathcal{H}_1$, each choice giving a \emph{different} unitary operator $W$ of the form $W(h)=\sum_{i \in I}\langle h, e_i\rangle_{\mathcal{H}}\xi_i$.
The subtlety in the use of Proposition \ref{Gaussian_isometries}, is that one has to select an Hilbert space $\mathcal{H}$ and a unitary operator $W$
that are well adapted to the specific problem at hand. 

\subsection{Wiener Chaos Decomposition}

Every Gaussian Hilbert space induces an orthogonal decomposition, known as the Wiener Chaos Decomposition, of the corresponding $L^2(\Omega,\sigma(\mathcal{H}_1),\mathbb{P})$ space of square integrable random variables that are measurable with respect to the $\sigma$-field generated by the Gaussian Hilbert space, that we denote by $\sigma(\mathcal{H}_1)$. For $n \ge 0$ we introduce the linear space
\begin{equation*}
\mathcal{P}_n(\mathcal{H}_1):= \left\{p(\xi_1,\ldots,\xi_m)\,|\, \ p \text{ is a polynomial of degrees $\le n$}, \ \xi_1,\ldots,\xi_m \in \mathcal{H}_1, \ m\in\N \right\}.
\end{equation*}
Let $\overline{\mathcal{P}_n(\mathcal{H}_1)}$ be the closure of $\mathcal{P}_n(\mathcal{H}_1)$ in $L^2(\Omega,\sigma(\mathcal{H}_1),\mathbb{P})$.
For $n \ge 0$ the space
\begin{equation*}
\mathcal{H}_n:= \overline{\mathcal{P}_n(\mathcal{H}_1)} \ominus  \overline{\mathcal{P}_{n-1}(\mathcal{H}_1)} =  \overline{\mathcal{P}_n(\mathcal{H}_1)} \cap  \overline{\mathcal{P}_{n-1}(\mathcal{H}_1)}^{\perp}
\end{equation*}
is called $n$-th Wiener Chaos (associated to $\mathcal{H}_1$). We remark that $\mathcal{H}_0=\R$. 
The following result is usually called Wiener chaos decomposition, its proof can be found in \cite[Theorem 2.6]{Janson}.

\begin{theorem}\label{WCD}
The spaces $\mathcal{H}_n$, $n \ge 0$, are mutually orthogonal, closed subspaces of $L^2(\Omega, \mathcal{F}, \mathbb{P})$ and 
\begin{equation*}
L^2(\Omega, \sigma(\mathcal{H}_1), \mathbb{P})=\bigoplus_{n=0}^\infty \mathcal{H}_n.
\end{equation*}
\end{theorem}

For $n \ge 0$ let us denote by $J_n$ the orthogonal projection of $L^2(\Omega,\sigma(\mathcal{H}_1),\mathbb{P})$ onto $\mathcal{H}_n$; in particular, $J_0(X)=\mathbb{E}\left[X\right]$. Theorem \ref{WCD} yields that every random variable $X \in L^2(\Omega, \sigma(\mathcal{H}_1), \mathbb{P})$ admits the unique expansion 
\begin{equation*}
X=\sum_{n=0}^{+\infty} J_n(X)=\mathbb{E}\left[X\right] + \sum_{n=1}^{+\infty} J_n(X),
\end{equation*}
with the series converging in $L^2(\Omega, \sigma(\mathcal{H}_1),\mathbb{P})$.


\subsection{Malliavin derivative operators and Sobolev spaces}
From here on we fix a probability space $(\Omega, \mathcal{F}, \mathbb{P})$ and an infinite dimensional separable Gaussian Hilbert space $\mathcal{H}_1$. We assume $\mathcal{F}$ to be the $\sigma$-field generated by $\mathcal{H}_1$.
Moreover, according to Proposition \ref{Gaussian_isometries}, we fix a separable Hilbert space $\mathcal{H}$ and a unitary operator
\begin{equation*}
W:\mathcal{H} \rightarrow \mathcal{H}_1\subseteq L^2(\Omega, \sigma(\mathcal{H}_1), \mathbb{P}),
\end{equation*}
so that we characterize 
\begin{equation*}
\mathcal{H}_1=\{W(h)\,|\, h \in \mathcal{H}\},
\end{equation*}
and every $W(h)\in \mathcal{H}_1$ is a centered Gaussian random variable with variance 
\[\|W(h)\|^2_{L^2(\Omega, \sigma(\mathcal{H}_1), \mathbb{P})}=\|h\|^2_{\mathcal{H}}.\] 
Let us denote by $\mathcal{S}(\mathcal{H}_1)$ the set of smooth random variables, i.e. random variables of the form 
\begin{equation}\label{MariaClara}
F=f(W(h_1),\ldots,W(h_m))
\end{equation}
for some $m \ge 1$ and $h_1,\ldots,h_m \in \mathcal{H}$, where $f$ is a $C^\infty(\mathbb{R}^m)$ function such that $f$ and all its partial derivatives have at most polynomial growth.

\begin{definition}\label{defDpol}
The derivative of a random variable $F \in \mathcal{S}(\mathcal{H}_1)$ of the form \eqref{MariaClara} is the $\mathcal{H}$-valued random variable
\begin{equation*}
DF=\sum_{i=1}^m \frac{\partial f}{\partial x_i} (W(h_1),\ldots,W(h_m)) h_i.
\end{equation*}
\end{definition}

The space $\mathcal{S}(\mathcal{H}_1)$ turns out to be dense $L^p(\Omega, \sigma(\mathcal{H}_1), \mathbb{P})$ for any $p \in [1, +\infty)$, see e.g. \cite[Lemma 3.2.1]{NP}. This, along with the following integration by parts formula (see e.g. \cite[Lemma 1.2.1]{Nualart}):
\begin{equation}
\label{ibp_0}
\mathbb{E}\left[ \langle DF, h \rangle_{\mathcal{H}}\right]=\mathbb{E}\left [W(h)F\right], \qquad  h\in \mathcal{H}, \ F \in \mathcal{S}(\mathcal{H}_1),
\end{equation}
is the crucial ingredient to extend the class of differential random variables to a larger class. For a proof of the following proposition see \cite[Proposition 2.3.4]{NP}.

\begin{proposition}
\label{closableProp}
For any $p \in [1, +\infty)$ the operator 
\[D: \mathcal{S}(\mathcal{H}_1)\subseteq L^p(\Omega, \sigma(\mathcal{H}_1),\mathbb{P}) \rightarrow L^p(\Omega, \sigma(\mathcal{H}_1),\mathbb{P}; \mathcal{H}),\] 
introduced in Definition \ref{defDpol}, is closable as an operator from the space $L^p(\Omega, \sigma(\mathcal{H}_1),\mathbb{P})$ to the space $L^p(\Omega, \sigma(\mathcal{H}_1),\mathbb{P}; \mathcal{H})$.
\end{proposition}

For any $p \in [1, +\infty)$ we denote with $\mathbb{D}^{1,p}$ the closure of $\mathcal{S}(\mathcal{H}_1)$ with respect to the norm
\begin{equation}
\label{norm_D^1p}
\|F\|^p_{\mathbb{D}^{1,p}} =\mathbb{E}\left[|F|^p\right]+\mathbb{E}\left[\|DF\|^p_{\mathcal{H}}\right].
\end{equation} 
According to Proposition \ref{closableProp} the operator $D$ 
admits a closed extension (still denoted by $D$) with domain $\mathbb{D}^{1,p}$. We call this extension 
\emph{Malliavin derivative} and we call $\mathbb{D}^{1,p}$ the \emph{domain of $D$} in $L^p(\Omega, \sigma(\mathcal{H}_1), \mathbb{P})$.
For any $p \in [1, +\infty)$ the space $\mathbb{D}^{1,p}$ endowed with the norm \eqref{norm_D^1p} is a Banach space, for $p=2$ the space $\mathbb{D}^{1,2}$ is a Hilbert space with the inner product 
\begin{equation*}
\langle F, G\rangle_{\mathbb{D}^{1,2}}= \mathbb{E}\left[FG\right] + \mathbb{E}\left[ \langle DF, DG\rangle_{\mathcal{H}}\right].
\end{equation*}
It is not difficult to prove that the integration by parts formula \eqref{ibp_0} extends to elements in $\mathbb{D}^{1,2}$, that is 
\begin{equation}
\label{ibp_1}
\mathbb{E}\left[ \langle DF, h \rangle_{\mathcal{H}}\right]=\mathbb{E}\left [W(h)F\right], \qquad h\in \mathcal{H}, \ F \in \mathbb{D}^{1,2}.
\end{equation}

The space $\mathbb{D}^{1,2}$ is characterized in the following proposition, in terms of the Wiener chaos expansion (see \cite[Proposition 1.2.2]{Nualart}). 
\begin{proposition}
\label{charD12}
Let $F \in L^2(\Omega, \sigma(\mathcal{H}_1), \mathbb{P})$ with Wiener chaos expansion $F=\sum_{n=0}^\infty J_n(F)$. Then $F \in \mathbb{D}^{1,2}$ if, and only if, 
\begin{equation*}
\mathbb{E}\left[ \|DF\|^2_{\mathcal{H}}\right] = \sum_{n=1}^\infty n\|J_n(F)\|_{L^2(\Omega)}^2 <\infty.
\end{equation*}
\end{proposition}

Let us emphasize that, once we have fixed the reference probability space $(\Omega, \mathcal{F}, \mathbb{P})$ and the Gaussian Hilbert spaces $\mathcal{H}_1$,
different (infinitely many) choices of the separable Hilbert space $\mathcal{H}$ and the unitary operator $W$ lead to different (infinitely many!) Malliavin derivative operators. 
On the other hand, in view of Proposition \ref{charD12}, all these Malliavin derivatives have the same domain $\mathbb{D}^{1,2}$ when the Gaussian Hilbert space $\mathcal{H}_1$ is the same. In fact the characterization of $\mathbb{D}^{1,2}$ is given in terms of the Wiener chaos decomposition that relies only on the Gaussian Hilbert space $\mathcal{H}_1$ (and not on the choices of $\mathcal{H}$ and $W$). 

%
%
%
%
%
%
%

\section*{Declarations}

\subsection*{Acknowledgments} The authors are grateful to E. Priola and L. Tubaro for numerous useful comments and discussions. 

\subsection*{Fundings} The authors are members of GNAMPA (Gruppo Nazionale per l'Analisi Ma\-te\-ma\-ti\-ca, la Probabilit\`a e le loro Applicazioni) of the Italian Istituto Nazionale di Alta Ma\-te\-ma\-ti\-ca (INdAM). D. A. Bignamini and S. Ferrari have been partially supported by the
INdAM-GNAMPA Project 2023 ``Equazioni differenziali stocastiche e operatori di Kolmogorov in
dimensione infinita'' CUP\_E53C22001930001. M. Zanella has been partially supported by the
INdAM-GNAMPA Project 2023 ``Analisi qualitativa di PDE e PDE stocastiche per modelli fisici'' CUP\_E53C22001930001. The authors have no relevant financial or non-financial interests to disclose.

\subsection*{Research Data Policy and Data Availability Statements} Data sharing not applicable to this article as no datasets were generated or analysed during the current study.

%



\begin{thebibliography}{99}

 

\bibitem{Add2021} 
Addona, D. (2021). Analyticity of nonsymmetric Ornstein--Uhlenbeck semigroup with respect to a weighted Gaussian measure. \emph{Potential Anal.} \textbf{54}(1), 95--117.

\bibitem{Add-Ban-Mas2020}
Addona, D., E. Bandini and F. Masiero (2020). A nonlinear Bismut--Elworthy formula for HJB equations with quadratic Hamiltonian in Banach spaces. \emph{NoDEA Nonlinear Differential Equations Appl.} \textbf{27}(4), Paper No. 37, 56.

\bibitem{ACF20}
Addona, D., G. Cappa and S. Ferrari (2020). Domains of elliptic operators on sets in Wiener space. \emph{Infin. Dimens. Anal. Quantum Probab. Relat. Top.} \textbf{23}(1), 2050004.

\bibitem{ACF22}
Addona, D., G. Cappa and S. Ferrari (2022). On the domain of non-symmetric and, possibly, degenerate Ornstein--Uhlenbeck operators in separable Banach spaces. \emph{Rend. Lincei Mat. Appl.} \textbf{33}(2), 297--359.

\bibitem{Add-Mas-Pri2023} 
Addona, D., F. Masiero and E. Priola (2023). A BSDEs approach to pathwise uniqueness for stochastic evolution equations. \emph{J. Differential Equations} \textbf{366}, 192--248.

\bibitem{Add-Men-Mir2020} 
Addona, D., G. Menegatti and M. Miranda jr. (2020). BV functions on open domains: the Wiener case and the Fomin differentiable case. \emph{Commun. Pure Appl. Anal.} \textbf{19}(5) (2020), 2679--2711.

\bibitem{Addo1}
Addona, D., G. Menegatti and M. Miranda jr. (2021). On integration by parts formula on open convex sets in Wiener spaces. \emph{J. Evol. Equ.} \textbf{21}(2), 1917--1944.

\bibitem{Add-Men-Mir2023}  
Addona, D., G. Menegatti and M. Miranda jr. (2023). Characterizations of Sobolev spaces on sublevel sets in abstract Wiener spaces. \emph{J. Math. Anal. Appl.} \textbf{524}(1), Paper No. 127075, 20.

\bibitem{Add-Mur-Ros2022} 
Addona, D., M. Muratori and M. Rossi (2022). On equivalence of Sobolev norms in Malliavin spaces. \emph{J. Funct. Anal.} \textbf{283}(7), 109600.


\bibitem{ABF21} 
Angiuli, L., D. A. Bignamini and S. Ferrari (2023). Harnack inequalities with power {$p\in (1,+\infty )$} for transition semigroups in {H}ilbert spaces. \emph{NoDEA Nonlinear Differential Equations Appl.} \textbf{30}(1), Paper No. 6, 30.

\bibitem{AL21}
Al\`os, E. and D. Garc\'\i a Lorite (2021). \textit{{M}alliavin {C}alculus in {F}inance: Theory and Practice}. Volume of \emph{Chapman} \& \emph{Hall/CRC financial mathematics series}. Boca Raton, FL: CRC Press, Taylor \& Francis Group.

\bibitem{Bally1998}
Bally, V. and E. Pardoux (1998). Malliavin calculus for white noise driven parabolic {SPDE}s. \emph{Potential Anal.} \textbf{9}(1), 27--64.

\bibitem{Talay1}
Bally, V. and D. Talay (1996). The law of the {E}uler scheme for stochastic differential equations: {I}. {C}onvergence rate of the distribution function. \emph{Probab. Theory Relat. Fields} \textbf{104}(1), 43--60.

\bibitem{Talay2}
Bally, V. and D. Talay (1996). The law of the {E}uler scheme for stochastic differential equations: {II}. {C}onvergence rate of the density. \emph{Monte Carlo Methods Appl.} \textbf{2}(2), 93--128.

\bibitem{Big21} 
Bignamini, D. A. (2023). $L^2$-theory for transitions semigroups associated to dissipative systems. \emph{Stoch. PDE: Anal. Comp.} \textbf{11}(3), 988--1043.

\bibitem{BF22} 
Bignamini, D. A. and S. Ferrari (2022). On generators of transition semigroups associated to semilinear stochastic partial differential equations. \emph{J. Math. Anal. Appl.} \textbf{508}(1), Paper No. 125878, 40.

\bibitem{BF20} 
Bignamini, D. A. and S. Ferrari (2023). Regularizing properties of (non-{G}aussian) transition semigroups in {H}ilbert spaces. \emph{Potential Anal.} \textbf{58}(1), 1--35.

\bibitem{BFS} 
Bignamini, D. A. and S. Ferrari (2023). Schauder regularity results in separable {H}ilbert spaces. \emph{J. Differential Equations} \textbf{370}, 305--345.

\bibitem{BFS2} 
Bignamini, D. A. and S. Ferrari (2024). Schauder estimates for stationary and evolution equations associated to stochastic reaction-diffusion equations driven by colored noise. \emph{Stoch. Anal. Appl.} \textbf{DOI}: 10.1080/07362994.2024.2303099.


\bibitem{Bis1}
Bismut, J.-M. (1981). Martingales, the {M}alliavin calculus and hypoellipticity under general {H}\"{o}rmander's conditions. \emph{Z. Wahrscheinlichkeitstheorie verw. Gebiete} \textbf{56}(4) 469--505.

\bibitem{Bis2}
Bismut, J.-M. (1984). \textit{Large deviations and the {M}alliavin calculus}. Volume \textbf{45} of \emph{Progress in Mathematics}. Boston, MA: Birkh\"{a}user Boston, Inc.

\bibitem{BOGIE1} 
Bogachev, V. I. (1998). \textit{Gaussian measures}. Volume \textbf{62} of \emph{Mathematical Surveys and Monographs}. Providence, RI: American Mathematical Society.

\bibitem{BOG18}
Bogachev, V. I. (2018). Ornstein--Uhlenbeck operators and semigroups. \emph{Russ. Math. Surv.}
\textbf{73}, 191--260.

\bibitem{BDPT}
Bonaccorsi, S., G. Da Prato and L. Tubaro (2018). Construction of a surface integral under local {M}alliavin assumptions, and related integration by parts formulas. \emph{J. Evol. Equ.} \textbf{18}(2), 871--897.

\bibitem{BoZa2}
Bonaccorsi, S. and M. Zanella (2016). Existence and regularity of the density for solutions of stochastic differential equations with boundary noise. \emph{Infin. Dimens. Anal. Quantum Probab. Relat. Top.} \textbf{19}(1), 1650007, 24.

\bibitem{BoZa1}
Bonaccorsi, S. and M. Zanella (2017). Absolute continuity of the law for solutions of stochastic differential equations with boundary noise. \emph{Stoch. Dyn.} \textbf{17}(6), 1750045, 31.

\bibitem{BoTuZa}
Bonaccorsi, S., L. Tubaro and M. Zanella (2020). Surface measures and integration by parts formula on levels sets induced by functionals of the {B}rownian motion in $\mathbb{R}^n$. \emph{NoDEA Nonlinear Differential Equations Appl.} \textbf{27}(3), Paper No. 27, 22.

\bibitem{CDP96} 
Cannarsa, P. and G. Da Prato (1996). Infinite-dimensional elliptic equations with {H}\"{o}lder-continuous coefficients. \emph{Adv. Differential Equations} {\bf 1}(3), 425--452.

\bibitem{CDP96-2} 
Cannarsa, P. and G. Da Prato (1996). Schauder estimates for {K}olmogorov equations in {H}ilbert spaces. In A. Alvino, P. Buonocore, V. Ferone, E. Giarrusso, S. Matarasso, R. Toscano and G. Trombetti (Eds.), \emph{Progress in elliptic and parabolic partial differential equations ({C}apri, 1994)}, Volume \textbf{350} of \emph{Pitman Res. Notes Math. Ser.}, pp. 100--111.  Harlow: Longman.

\bibitem{CF16} 
Cappa, G. and S. Ferrari (2016). Maximal Sobolev regularity for solutions of elliptic equations in infinite dimensional Banach spaces endowed with a weighted Gaussian measure. \emph{J. Differential Equations} \textbf{261}(12), 7099--7131.

\bibitem{CF18} 
Cappa, G. and S. Ferrari (2018). Maximal Sobolev regularity for solutions of elliptic equations in Banach spaces endowed with a weighted Gaussian measure: The convex subset case. \emph{J. Math. Anal. Appl.} \textbf{458}(1), 300--331.

\bibitem{Cardon-Weber2001}
Cardon-Weber, C. (2001). Cahn--{H}illiard stochastic equation: existence of the solution and of its density. \emph{Bernoulli} \textbf{7}(5), 777--816.

\bibitem{CDP12} 
Cerrai, S. and G. Da Prato (2012). Schauder estimates for elliptic equations in Banach spaces associated with stochastic reaction-diffusion equations. \emph{J. Evol. Equ.} {\bf 12}, 83--98.

\bibitem{CL19} 
Cerrai, S. and A. Lunardi (2019). Schauder theorems for {O}rnstein--{U}hlenbeck equations in infinite dimension. \emph{J. Differential Equations} {\bf 267}(12), 7462--7482.

\bibitem{CMG96}
Chojnowska-Michalik, A. and B. Goldys (1996). Nonsymmetric Ornstein--Uhlenbeck semigroup as second quantized operator. \emph{J. Math. Kyoto Univ.} \textbf{36}, 481--498.

\bibitem{Crisan}
Crisan, D., K. Manolarakis and N. Touzi (2010). On the {M}onte {C}arlo simulation of {BSDE}s: an improvement on the {M}alliavin weights. \emph{Stoch. Process. Their Appl.} \textbf{120}(7), 1133--1158.

\bibitem{DKN2009}
Dalang, R. C., D. Khoshnevisan and E. Nualart (2009). Hitting probabilities for systems for non-linear stochastic heat equations with multiplicative noise. \emph{Probab. Theory Relat. Fields} \textbf{144}(3-4), 371--427.

\bibitem{DaPrato13} 
Da Prato, G. (2013). Schauder estimates for some perturbation of an infinite dimensional Ornstein--Uhlenbeck operator. \emph{Discrete Contin. Dyn. Syst. - S} {\bf 6}(3), 637--647.

\bibitem{DaPrato}
Da Prato, G. (2014). \textit{Introduction to stochastic analysis and Malliavin calculus} (3 ed). Volume \textbf{13} of {\em Appunti. Scuola Normale Superiore di Pisa (Nuova Serie) [Lecture Notes. Scuola Normale Superiore di Pisa (New Series)]}. Pisa: Edizioni della Normale.

\bibitem{DA-LU2} 
Da Prato, G. and A. Lunardi (2014). Sobolev regularity for a class of second order elliptic {PDE}'s in infinite dimension. \emph{Ann. Probab.} {\bf 42}(5), 2113--2160.

\bibitem{DA-LU3} 
Da Prato, G. and A. Lunardi (2015). Maximal {S}obolev regularity in {N}eumann problems for gradient systems in infinite dimensional domains. \emph{Ann. Inst. Henri Poincar\'{e} Probab. Stat.} {\bf 51}(3), 1102--1123.

\bibitem{DPLT2}
Da Prato, G., A. Lunardi and L. Tubaro (2014). Surface Measures In Infinite Dimension. \emph{Atti Accad. Naz. Lincei Rend. Lincei Mat. Appl.} \textbf{25}(3), 309--330.

\bibitem{DA-LU-TU1} 
Da Prato, G., A. Lunardi and L. Tubaro (2016). Malliavin Calculus for non Gaussian differentiable measures and surface measures in Hilbert spaces. \emph{Trans. Amer. Math. Soc.} {\bf 370}(8), 5795--5842.


\bibitem{DA-TU2} 
Da Prato, G. and L. Tubaro (2001). Some results about dissipativity of Kolmogorov operators. \emph{Czechoslovak Math. J.} \textbf{51(126)}(4), 685--699.


\bibitem{DaPratoZab_Second}
Da Prato, G. and J. Zabczyk (2002). \textit{Second order partial differential equations in Hilbert spaces}. Volume \textbf{293} of \emph{London Mathematical Society Lecture Note Series}. Cambridge: Cambridge University Press.

\bibitem{DaPratoZab_Stoch}
Da Prato, G. and J. Zabczyk (2014). \textit{Stochastic equations in infinite dimensions}. Volume \textbf{152} of \emph{Encyclopedia of Mathematics and its Applications}. Cambridge: Cambridge University Press.

\bibitem{DS88I}
Dunford, N. and J. T. Schwartz (1988). \textit{Linear operators. {P}art {I}}. Volume of \emph{Wiley Classics Library}. New York: John Wiley \& Sons, Inc.


\bibitem{FHH11}
Fabian, M., P. Habala, P. H\'{a}jek, V. Montesinos and V. Zizler (2011). \textit{Banach space theory}. Volume of \emph{CMS Books in Mathematics/Ouvrages de Math\'{e}matiques de la SMC}. New York: Springer.

\bibitem{FE-GO2} 
Federico, S. and F. Gozzi (2017). Corrigendum to ``{M}ild solutions of semilinear elliptic equations in {H}ilbert spaces'' [{J}. {D}ifferential {E}quations 262 (2017) 3343--3389] [{MR}3584895]. \emph{J. Differential Equations} \textbf{263}(9), 6143--6144.

\bibitem{FE-GO1}
Federico, S. and F. Gozzi (2017). Mild solutions of semilinear elliptic equations in {H}ilbert spaces. \emph{J. Differential Equations} \textbf{262}(5), 3343--3389.

\bibitem{FerZan}
Ferrario, B. and M. Zanella (2019). Absolute continuity of the law for the two dimensional stochastic {N}avier-{S}tokes equations. \emph{Stochastic Process. Appl.} \textbf{129}(5), 1568--1604.

\bibitem{Fra1}
Franz, U., R. L\'eandre and R. Schott (2001). Malliavin calculus and Skorohod integration for quantum stochastic processes. \emph{Infin. Dimens. Anal. Quantum Probab. Relat. Top.} \textbf{4}(1), 11--38. 

\bibitem{Fre80}
Flett, T. M. (1980). \textit{Differential Analysis}. Cambridge-New York: Cambridge University Press.

\bibitem{GT}
Gaveau, B. and P. Trauber (1982). L'int\'egrale stochastique comme op\'erateur de divergence dans l'espace fonctionnel. \emph{J. Funct. Anal.} \textbf{46}(2), 230--238.

\bibitem{GGvN03} 
Goldys, B., F. Gozzi and J. M. A. M. van Neerven (2003). On closability of directional gradients. \emph{Potential Anal.} {\bf 18}, 289--310.

\bibitem{GRO1}  
Gross, L. (1967). Potential theory on Hilbert space. \emph{J. Funct. Anal.} {\textbf{1}}, 123--181.

\bibitem{GOZ1}  
Gozzi, F. (2006). Smoothing properties of nonlinear transition semigroups: case of {L}ipschitz nonlinearities. \emph{J. Evol. Equ.} \textbf{6}(4), 711--743.

\bibitem{HaiMat}
Hairer, M. and J. C. Mattingly (2006). Ergodicity of the 2{D} {N}avier-{S}tokes equations with degenerate stochastic forcing. \emph{Ann. of Math. (2)} \textbf{164}(3), 993--1032.

\bibitem{Janson}
Janson, S. (1997). \textit{Gaussian {H}ilbert spaces}. Volume \textbf{129} of \emph{Cambridge Tracts in Mathematics}. Cambridge: Cambridge University Press.

\bibitem{KE1} 
Kechris, A. S. (1995). \textit{Classical descriptive set theory}. Volume \textbf{156} of \emph{Graduate Texts in Mathematics}. New York: Springer-Verlag.

\bibitem{KUO1}  
Kuo, H. H. (1975). \textit{Gaussian measures in Banach space}. Volume \textbf{463} of \emph{Lecture Notes in Mathematics}. Berlin-New York: Springer-Verlag.

\bibitem{KS84}
Kusuoka, S. and D. Stroock (1984). Applications of the {M}alliavin calculus, Part {I}.  In K. It\^o (Eds.) \emph{Stochastic analysis. Proceedings of the Taniguchi International Symposium on Stochastic Analysis, Katata and Kyoto, 1982}, Volume \textbf{32} of \emph{North-Holland Mathematical Library}, pp. 271--306. Amsterdam: North-Holland.

\bibitem{KS85}
Kusuoka, S. and D. Stroock (1985). Applications of the {M}alliavin calculus, Part {II}. \emph{J. Fac. Sci. Univ. Tokyo Sect. IA, Math.} \textbf{32}(1), 1--76.

\bibitem{KS87}
Kusuoka, S. and D. Stroock (1987). Applications of the {M}alliavin calculus, Part {III}. \emph{J. Fac. Sci. Univ. Tokyo Sect. IA, Math.} \textbf{34}(2), 391--442.

\bibitem{NualartZaidi1997}
Lanjri Zaidi, N. and D. Nualart (1999). Burgers equation driven by a space-time white noise: absolute continuity of the solution. \emph{Stoch. Stoch. Rep.} \textbf{66}(3-4), 273--292.

\bibitem{LL86} 
Lasry, J.-M. and P.-L. Lions (1986). A remark on regularization in {H}ilbert spaces. \emph{Israel J. Math.} {\bf 55}(3), 257--266.

\bibitem{LI-RO1}
Liu, W. and M. R\"ockner (2015). \textit{Stochastic partial differential equations: an introduction}. Volume of \emph{Universitext}. Cham: Springer.


\bibitem{Lun18}
Lunardi, A. (2018) \textit{Interpolation theory} (3 ed). Volume \textbf{16} of \emph{Appunti. Scuola Normale Superiore di Pisa (Nuova Serie) [Lecture Notes. Scuola Normale Superiore di Pisa (New Series)]}. Pisa: Edizioni della Normale.

\bibitem{LMP20} 
Lunardi, A., G. Metafune and D. Pallara (2020). The Ornstein--Uhlenbeck semigroup in finite dimension. \emph{Phil. Trans. R. Soc. A} \textbf{378}, 20200217.

\bibitem{Lunardi}
Lunardi, A., M. Miranda jr. and D. Pallara (2016).
\textit{19th Internet Seminar: Infinite Dimensional Analysis: Lecture notes}. This lecture notes will be expanded into a book in the coming years.

\bibitem{LR21} 
Lunardi, A. and M. R\"{o}ckner (2021). Schauder theorems for a class of (pseudo-)differential operators on finite- and infinite-dimensional state spaces. \emph{J. Lond. Math. Soc. (2)} \textbf{104}(2), 492--540.

\bibitem{LP20} 
Lunardi, A. and D. Pallara (2020). Ornstein--Uhlenbeck semigroups in infinite dimension. \emph{Phil. Trans. R. Soc. A} \textbf{378}, 20190620.

\bibitem{Mal78}
Malliavin, P. (1978). Stochastic calculus of variation and hypoelliptic operators. In K. It\^o (Eds.) \emph{Stochastic Analysis, Kinokuniya, Tokyo}, pp. 195--263. New York-Chichester-Brisbane: John Wiley \& Sons, Wiley-Intersci. Publ.

\bibitem{Malliavin}
Malliavin, P. (1997). \textit{Stochastic analysis}. Volume \textbf{313} of \emph{Grundlehren der mathematischen Wissenschaften [Fundamental Principles of Mathematical Sciences]}. Berlin: Springer-Verlag.

\bibitem{Marinelli2013}
Marinelli, C, E. Nualart and L. Quer-Sardanyons (2013). Existence and regularity of the density for solutions to semilinear dissipative parabolic {SPDE}s. \emph{Potential Anal.} \textbf{39}(3), 287--311.

\bibitem{Marquez-Carreras2001}
M{\'a}rquez-Carreras, D., M. Mellouk and M. Sarr{\`a} (2001). On stochastic partial differential equations with spatially correlated noise: smoothness of the law. \emph{Stoch. Process. Their Appl.} \textbf{93}(2), 269--284.

\bibitem{MAS2} 
Masiero, F. (2005). Semilinear {K}olmogorov equations and applications to stochastic optimal control. \emph{Appl. Math. Optim.} \textbf{51}(2), 201--250.

\bibitem{MAS1} 
Masiero, F. (2007). Regularizing properties for transition semigroups and semilinear parabolic equations in {B}anach spaces. \emph{Electron. J. Probab.} \textbf{12}, no. 13, 387--419.

\bibitem{MA-PR2} 
Masiero, F. and E. Priola (2016). Correction to ``Well-posedness of semilinear stochastic wave equations with H\"older continuous coefficients''. e-print arXiv:1607.00029.

\bibitem{MA-PR1} 
Masiero, F. and E. Priola (2017). Well-posedness of semilinear stochastic wave equations with H\"older continuous coefficients. \emph{J. Differential Equations} \textbf{263}(3), 1773--1812.

\bibitem{Mey2006}
Meyer, P.-A. (2006). \textit{Quantum Probability for Probabilists}. Volume \textbf{1538} of \emph{Lecture Notes in Mathematics}. Heidelberg: Springer Berlin.

\bibitem{MilletSanz}
Millet, A. and M. Sanz-Sol\'e (1999). A stochastic wave equation in two space dimension: smoothness of the law. \emph{Ann. Probab.} \textbf{27}(2), 803--844.

\bibitem{Morien1999}
Morien, P.-L. (1999). On the density for the solution of a {B}urgers-type {SPDE}. \emph{Ann. Inst. Henri Poincar\'{e} Probab. Stat.} \textbf{35}(4), 459--482.

\bibitem{Mueller2008}
Mueller, C. and D. Nualart (2008). Regularity of the density for the stochastic heat equation. \emph{Electron. J. Probab.} \textbf{13}, 2248--2258.

\bibitem{NP}  
Nourdin, I. and G. Peccati (2012). \textit{Normal approximations with {M}alliavin calculus}. Volume \textbf{192} of \emph{Cambridge Tracts in Mathematics}. Cambridge: Cambridge University Press.

\bibitem{NV2009}
Nourdin, I. and F. Viens (2009). Density formula and concentration inequalities with {M}alliavin calculus. \emph{Electron. J. Probab.} \textbf{14}, 2287--2309.

\bibitem{Nualart}
Nualart, D. (2006). \textit{The {M}alliavin calculus and related topics} (2 ed). Volume of \emph{Probability and its Applications (New York)}. Berlin: Springer-Verlag.

\bibitem{Sard_Nualart}
Nualart, D. and L. Quer-Sardanyons (2007). Existence and smoothness of the density for spatially homogeneous {SPDE}s. \emph{Potential Anal.} \textbf{27}(3), 281--299.

\bibitem{Pardoux1993}
Pardoux, E. and T. S. Zhang (1993). Absolute continuity of the law of the solution of a parabolic {SPDE}. \emph{J. Funct. Anal.} \textbf{112}(2), 447--458.

\bibitem{PRI0}
Priola, E. (1998). $\pi$-Semigroups and applications. \emph{Scuola Norm. Sup. Pisa}, Preprint n. 9.

\bibitem{PRI1}
Priola, E. (1999). \textit{Partial differential equations with infinitely many variables}. Iris, AperTO, Universit\`a degli Studi di Torino.

\bibitem{Pri99} 
Priola, E. (1999). On a class of {M}arkov type semigroups in spaces of uniformly continuous and bounded functions. \emph{Studia Math.} {\bf 136}(3), 271--295.

\bibitem{PZ00} 
Priola, E. and L. Zambotti (2000). New optimal regularity results for infinite-dimensional elliptic equations. \emph{Boll. Unione Mat. Ital. Sez. B Artic. Ric. Mat. (8)} {\bf 3}(2), 411--429.

\bibitem{Sar_San2}
Quer-Sardanyons, L. and M. Sanz-Sol\'e (2004). A stochastic wave equation in dimension 3: smoothness of the law. \emph{Bernoulli} \textbf{10}(1), 165--186.

\bibitem{Sanz_Sardanyons}
Quer-Sardanyons. L. and M. Sanz-Sol{\'e} (2004). Absolute continuity of the law of the solution to the 3-dimensional stochastic wave equation. \emph{J. Funct. Anal.} \textbf{206}(1), 1--32.

\bibitem{RE-SI1}
Reed, M. and B. Simon (1972). \textit{Methods of modern mathematical physics. {I}. {F}unctional analysis}. New York-London: Academic Press.

\bibitem{Shi1}
Shigekawa, I. (1980). Derivatives of {W}iener functionals and absolute continuity of induced measures. \emph{J. Math. Kyoto Univ.} \textbf{20}(2), 263--289.

\bibitem{Shi2}
Shigekawa, I. (2004). \textit{Stochastic analysis}. Translations of Mathematical Monographs, volume \textbf{224} of \emph{Iwanami Series in Modern Mathematics}. Providence, RI: American Mathematical Society.
  
\bibitem{Stroock}
Stroock, D. W. (1981). The {M}alliavin calculus, a functional analytic approach. \emph{J. Funct. Anal.} \textbf{44}(2), 212--257.

\bibitem{Tri95}
Triebel, H. (1995). \textit{Interpolation theory, function spaces, differential operators} (2 ed). Heidelberg: Johann Ambrosius Barth.

\bibitem{TZ}
Tubaro, L. and M. Zanella (2024). \textit{An Introduction to Malliavin calculus}. Lecture notes, in preparation.

\bibitem{Roz}
Wan, X., B. Rozovskii and Karniadakis, G. E. (2009). A stochastic modeling methodology based on weighted {W}iener chaos and {M}alliavin calculus. \emph{Proc. Natl. Acad. Sci. U.S.A.} \textbf{106}(34), 14189--14194.

\bibitem{W84}
Watanabe, S. (1984). \textit{Lectures on stochastic differential equations and {M}alliavin calculus}. Volume \textbf{73} of \emph{Tata Institute of Fundamental Research Lectures on Mathematics and Physics}. Tata Institute of Fundamental Research, Bombay: Springer-Verlag

\bibitem{Zakai}
Zakai, M. (1985). The {M}alliavin calculus. \emph{Acta Appl. Math.} \textbf{3}(2), 175--207.

\end{thebibliography}
\end{document}